\documentclass[a4paper,11pt]{article}
\usepackage{amsmath, amsfonts, amscd, amssymb, amsthm, enumerate, xypic}

\def\defterm{\textbf}

\def\card{\#\,}
\newcommand{\Mat}{\operatorname{M}}
\newcommand{\Mata}{\operatorname{A}}
\newcommand{\NT}{\operatorname{NT}}

\newcommand{\GL}{\operatorname{GL}}
\newcommand{\Ker}{\operatorname{Ker}}
\newcommand{\Vect}{\operatorname{span}}
\newcommand{\im}{\operatorname{Im}}
\newcommand{\tr}{\operatorname{tr}}
\newcommand{\Sp}{\operatorname{Sp}}

\renewcommand{\setminus}{\smallsetminus}

\def\F{\mathbb{F}}
\def\K{\mathbb{K}}
\def\R{\mathbb{R}}
\def\C{\mathbb{C}}

\def\calB{\mathcal{B}}

\def\calJ{\mathcal{J}}

\def\calV{\mathcal{V}}
\def\calW{\mathcal{W}}

\def\lcro{\mathopen{[\![}}
\def\rcro{\mathclose{]\!]}}

\theoremstyle{definition}
\newtheorem{Def}{Definition}
\newtheorem{Not}[Def]{Notation}

\theoremstyle{plain}
\newtheorem{theo}{Theorem}
\newtheorem{prop}[theo]{Proposition}
\newtheorem{cor}[theo]{Corollary}
\newtheorem{lemme}[theo]{Lemma}
\newtheorem{claim}{Claim}

\theoremstyle{plain}

\theoremstyle{remark}
\newtheorem{Rems}{Remarks}
\newtheorem{Rem}[Rems]{Remark}

\title{Large affine spaces of non-singular matrices}
\author{Cl\'ement de Seguins Pazzis\footnote{Lyc\'ee Priv\'e Sainte-Genevi\`eve, 2, rue
de l'\'Ecole des Postes, 78029 Versailles Cedex, FRANCE.}
\footnote{e-mail address: dsp.prof@gmail.com}}

\begin{document}

\thispagestyle{plain}

\maketitle
\begin{abstract}
Let $\K$ be an arbitrary (commutative) field with at least three elements.
It was recently proven that an affine subspace of $\Mat_n(\K)$
consisting only of non-singular matrices must have a dimension lesser than or equal to $\binom{n}{2}$.
Here, we classify, up to equivalence, the subspaces whose dimension equals $\binom{n}{2}$.
This is done by classifying, up to similarity, all the $\binom{n}{2}$-dimensional linear subspaces of $\Mat_n(\K)$
consisting of matrices with no non-zero invariant vector, reinforcing a classical theorem of Gerstenhaber.
Both classifications only involve the quadratic structure of the field $\K$.
\end{abstract}



\section{Introduction}

\subsection{Introduction and basic definitions}\label{pureintro}

In this article, we let $\K$ be an arbitrary (commutative) field. We denote by
$\Mat_n(\K)$ the algebra of square matrices with $n$ rows
and entries in $\K$, and by $\GL_n(\K)$ its group of invertible elements.
We also denote by $\Mat_{n,p}(\K)$ the vector space of matrices with $n$ rows, $p$ columns and entries in $\K$.
The transpose of a matrix $M$ is denoted by $M^T$. \\
An \textbf{affine} subspace $\calV$ of $\Mat_n(\K)$ is the translate of a linear subspace $V$ of $\Mat_n(\K)$:
then $V$ is uniquely determined by $\calV$ (it is the set of all matrices $M$ such that $M+\calV=\calV$)
and is called the \textbf{translation vector space} of $\calV$. \\
Given two linear (or affine) subspaces $V$ and $W$ of $\Mat_n(\K)$, we say that $V$ and $W$ are equivalent,
and we write $V \sim W$, if $W=PVQ$ for some $(P,Q)\in \GL_n(\K)^2$;
we say that $V$ and $W$ are similar, and we write $V \simeq W$, if $W=PVP^{-1}$ for some $P \in \GL_n(\K)$. \\
Two matrices $A$ and $B$ of $\Mat_n(\K)$ are called \textbf{congruent}, and we write
$A \approx B$, if $A=PBP^T$ for some $P \in \GL_n(\K)$.
Finally, two quadratic forms $q$ and $q'$ on vector spaces over $\K$ are called \textbf{similar}
if $q'$ is equivalent to $\lambda\,q$ for some $\lambda \in \K \setminus \{0\}$.

\vskip 3mm
Here, we are concerned with the geometry of $\GL_n(\K) \cup \{0\}$ as a cone in the vector space $\Mat_n(\K)$.
From the linear algebraist's viewpoint, the natural questions that one may ask are the following ones:
\begin{itemize}
\item What is the minimal linear (resp.\ affine) subspace of $\Mat_n(\K)$ containing $\GL_n(\K)$?
\item What is the minimal linear subspace of $\Mat_n(\K)$ containing $\Mat_n(\K) \setminus \GL_n(\K)$?
\item What are the maximal linear (resp.\ affine) subspaces included in $\GL_n(\K) \cup \{0\}$?
\item What are the maximal linear (resp.\ affine) subspaces included in $\Mat_n(\K) \setminus \GL_n(\K)$?
\end{itemize}
The first two problems have easy answers: $\GL_n(\K)$ always spans $\Mat_n(\K)$,
the affine subspace it generates is $\Mat_n(\K)$ unless $n=1$ and $\card \K=2$,
and $\Mat_n(\K) \setminus \GL_n(\K)$ spans $\Mat_n(\K)$ unless $n=1$. \\
The last two questions have no clear answer  however and depend widely on the field $\K$.
For example, $\GL_n(\C) \cup \{0\}$ contains no $2$-dimensional linear subspace, whilst
$\GL_{2n}(\R)\cup \{0\}$ always does. As for singular linear subspaces (i.e.\ linear subspaces included in $\Mat_n(\K) \setminus \GL_n(\K)$),
a classification of them is generally considered to be out of reach, even for an algebraically closed field,
although a lot of progress has been made in understanding their structure in the 1980's
(see the works of Atkinson, Lloyd and Stephens \cite{AtkLloyd1,AtkLloyd2,AtkLloyd3,AtkStep} and our recent \cite{dSPclass}).

Rather than try to classify all the linear (resp.\ affine) subspaces contained in $\GL_n(\K)$ or $\Mat_n(\K) \setminus \GL_n(\K)$,
a more modest approach is to find the maximal dimension for such a subspace and to classify
the linear (resp.\ affine) subspaces with a maximal \emph{dimension}. To this day, this problem
has been almost entirely solved:
\begin{itemize}
\item A linear subspace included in $\GL_n(\K) \cup \{0\}$ has dimension at most $n$;
linear subspaces in $\GL_n(\K) \cup \{0\}$ with dimension $n$ correspond to the
structures of (possibly non-associative and non-unital) division algebras on $\K^n$ that are compatible with its vector space structure
(see e.g.\ the last section of \cite{dSPsinglin}).
Note that no such subspace exists when $n \geq 2$ and $\K$ is algebraically closed.
\item An affine subspace included in $\Mat_n(\K) \setminus \GL_n(\K)$
has dimension at most $n(n-1)$. If its dimension is $n(n-1)$, then it is equivalent
to the space of matrices with zero as last column or to its transpose
(unless $n=2$ and $\card \K = 2$ in which case there is an additional equivalence class).
This is a classical result of Dieudonn\'e \cite{Dieudonne} (see also \cite{dSPaffpres} for a simplified proof)
which may be used to classify the endomorphisms of the vector space $\Mat_n(\K)$ that stabilize $\GL_n(\K)$
(see \cite{dSPsinglin}).
\end{itemize}

Here, we will focus on the affine subspaces of $\Mat_n(\K)$ that are included in $\GL_n(\K)$.
Let $\calV$ be such a subspace, and choose $P \in \calV$. Then $P^{-1}\calV$
is also included in $\GL_n(\K)$, contains the identity matrix $I_n$ and has the same dimension as $\calV$.
Denoting by $H$ its translation vector space,
we see that $I_n-\lambda\,M \in \GL_n(\K)$ for every $\lambda \in \K$ and $M \in H$,
hence the linear subspace $H$ has the two following equivalent properties:
\begin{enumerate}[(i)]
\item For every $M \in H$, one has $\Sp(M) \subset \{0\}$, where $\Sp(M)$ denotes the set of eigenvalues of $M$
\emph{in the field} $\K$.
\item No matrix of $H$ possesses a non-zero invariant vector in $\K^n$.
\end{enumerate}

\begin{Def}
A linear subspace $H$ of $\Mat_n(\K)$ is said to have a \defterm{trivial spectrum} if
no matrix of $H$ possesses a non-zero invariant vector in $\K^n$.
\end{Def}

Note that for such a linear subspace $H$ with a trivial spectrum, the affine subspace
$I_n+H$ is included in $\GL_n(\K)$, and so is any subspace equivalent to it.
For example, if we denote by $\NT_n(\K)$ the space of strictly upper triangular matrices of $\Mat_n(\K)$, then
$I_n+\NT_n(\K)$ is an affine subspace of non-singular matrices with dimension $\binom{n}{2}$.

It follows that classifying up to equivalence the affine subspaces of non-singular matrices essentially amounts to classifying up to similarity
the linear subspaces of $\Mat_n(\K)$ with a trivial spectrum. In the case
$\K$ is algebraically closed, the linear subspaces with a trivial spectrum are the linear subspaces of nilpotent matrices:
a famous theorem of Gerstenhaber \cite{Gerstenhaber} states that the dimension of such a subspace is
bounded above by $\binom{n}{2}$ and that equality occurs only for subspaces similar to $\NT_n(\K)$.
It is only very recently that the upper bound $\binom{n}{2}$ has been shown to apply to linear subspaces with a trivial spectrum
for an arbitrary field (see the works of Quinlan \cite{Quinlan} and our own \cite{dSPlargerank}):

\begin{theo}\label{upperbound}
Let $V$ be a linear subspace of $\Mat_n(\K)$ with a trivial spectrum. \\
Then $\dim V \leq \binom{n}{2}$.
\end{theo}

\begin{Def}
A linear subspace of $\Mat_n(\K)$ with a trivial spectrum is called \textbf{maximal}\footnote{This should
not be confused with the concept of maximality in the set of linear subspaces with a trivial spectrum ordered by
the inclusion of subsets.}
if its dimension is $\binom{n}{2}$.
\end{Def}

Our aim is to classify the maximal linear subspaces of $\Mat_n(\K)$ with a trivial spectrum.
Unlike the case of nilpotent linear subspaces, the structure of the ground field $\K$ plays a large part
in this classification. For example, if there exists a polynomial $t^2-at-b \in \K[t]$ with degree two and no root in $\K$,
then the line spanned by the companion matrix $\begin{bmatrix}
0 & b \\
1 & a
\end{bmatrix}$ is obviously a maximal linear subspace of $\Mat_2(\K)$ with a trivial spectrum and it is not similar to $\NT_2(\K)$.
Another example is given by the space $\Mata_n(\R)$ of skew-symmetric real matrices, which has a trivial spectrum and
dimension $\binom{n}{2}$, although it is not similar to $\NT_n(\R)$ if $n \geq 2$.

\subsection{Reducibility}

\begin{Not}
Let $V$ and $W$ be respective subsets of $\Mat_n(\K)$ and $\Mat_p(\K)$.
Set
$$V \vee W :=\biggl\{\begin{bmatrix}
A & B \\
0 & C
\end{bmatrix} \mid (A,B,C)\in V \times \Mat_{n,p}(\K) \times W
\biggr\} \subset \Mat_{n+p}(\K).$$
\end{Not}

Note that if $V$ and $W$ are maximal linear subspaces with a trivial spectrum, then
$V \vee W$ is a linear subspace with a trivial spectrum and dimension $\binom{n}{2}+\binom{p}{2}+pn=\binom{n+p}{2}$,
hence it is maximal. Notice also that the composition law $\vee$ is associative.

\begin{Def}
A maximal linear subspace of $\Mat_n(\K)$ with a trivial spectrum is called \textbf{irreducible} if
the only linear subspaces of $\K^n$ it stabilizes are $\{0\}$ and $\K^n$ (and we call it reducible otherwise).
\end{Def}

Conversely, let $H$ be a maximal linear subspace of $\Mat_n(\K)$ with a trivial spectrum.
Assume that there is a $p \in \lcro 1,n-1\rcro$ such that $F:=\K^p \times \{0\}$ is stabilized by every matrix of $H$.
Then we may write every matrix of $H$ as
$$M=\begin{bmatrix}
f(M) & g(M) \\
0 & h(M)
\end{bmatrix} \quad \text{for some $(f(M),g(M),h(M))\in \Mat_p(\K) \times \Mat_{p,n-p}(\K) \times \Mat_{n-p}(\K)$.}$$
Therefore $V:=f(H)$ and $W:=h(H)$ are linear subspaces respectively of $\Mat_p(\K)$ and $\Mat_{n-p}(\K)$, each with a trivial
spectrum, and since
$$\binom{n}{2}=\dim H \leq \dim V+\dim W+\dim g(H) \leq \binom{p}{2}+\binom{n-p}{2}+p(n-p)=\binom{n}{2},$$
we find that both $V$ and $W$ are maximal. Hence $H \subset V \vee W$, and since the dimensions are equal,
we deduce that $H =V \vee W$.

Conjugating $H$ with an appropriate invertible matrix, this generalizes as follows:
if $H$ is not irreducible, then
$H \simeq V \vee W$ for some maximal linear subspaces $V$ and $W$ with trivial spectra.
This yields:

\begin{prop}
Let $H$ be a maximal linear subspace of $\Mat_n(\K)$ with a trivial spectrum.
Then there are irreducible maximal linear subspaces $V_1,\dots,V_p$ with trivial spectra such that
$$H \simeq V_1 \vee V_2 \vee \cdots \vee V_p.$$
\end{prop}

\noindent This suggests that we focus our attention on the irreducible maximal subspaces.

\subsection{Main theorems}\label{1.3}

Denote by $\Mata_n(\K)$ the set of \emph{alternate} matrices of $\Mat_n(\K)$, i.e.\
the skew-symmetric ones with a zero diagonal, i.e.\ the ones for which
$\forall X \in \K^n, \; X^TAX=0$.

\begin{Def}
A matrix $P \in \Mat_n(\K)$ is called \textbf{non-isotropic} if the quadratic form
$X \mapsto X^TPX$ is non-isotropic, i.e. $\forall X \in \K^n \setminus \{0\}, \; X^TPX \neq 0$.
\end{Def}

Notice, in that case, that $P$ is non-singular and that $P^{-1}$ is non-isotropic.
The subspace $P \Mata_n(\K)$ then has dimension
$\binom{n}{2}$ and has a trivial spectrum: indeed, given $A \in \Mata_n(\K)$ and $X \in \K^n$,
$$PAX=X \Rightarrow P^{-1}X=AX \Rightarrow X^TP^{-1}X=0 \Rightarrow X=0.$$

We may now state our main results.

\begin{theo}\label{irreducible}
Assume that $\# \K \geq 3$. Let $n$ be a positive integer.
Then the irreducible maximal linear subspaces of $\Mat_n(\K)$ with a trivial spectrum
are the subspaces of the form $P\Mata_n(\K)$ for a non-isotropic matrix $P\in \GL_n(\K)$.
\end{theo}

\begin{theo}[Classification theorem for maximal linear subspaces with a trivial spectrum]\label{classtrivialspectrum}
Assume that $\# \K \geq 3$.
Let $V$ be a maximal linear subspace of $\Mat_n(\K)$ with a trivial spectrum.
Then there is a list
$(P_1,\dots,P_p)\in \GL_{n_1}(\K) \times \cdots \times \GL_{n_p}(\K)$ of non-isotropic matrices such that
$$V \simeq P_1\Mata_{n_1}(\K) \vee \cdots \vee P_p\Mata_{n_p}(\K).$$
The integer $p$ is uniquely determined by $V$ and, for every $k \in \lcro 1,p\rcro$, the matrix $P_k$ is uniquely determined
by $V$ up to congruence and multiplication by a non-zero scalar. Moreover, given
another list $(Q_1,\dots,Q_p) \in \GL_{n_1}(\K) \times \cdots \times \GL_{n_p}(\K)$,
if $Q_k$ is congruent to a scalar multiple of $P_k$ for each $k \in \lcro 1,p\rcro$, then
$$V \simeq Q_1\Mata_{n_1}(\K) \vee \cdots \vee Q_p\Mata_{n_p}(\K).$$
\end{theo}

If $\K$ is quadratically closed, it follows that there is no irreducible maximal linear subspace of $\Mat_n(\K)$
with a trivial spectrum for $n \geq 2$. If $\K$ is finite (with at least three elements), then every $3$-dimensional quadratic form over $\K$ is isotropic,
hence $\Mat_n(\K)$ contains no irreducible maximal linear subspace with a trivial spectrum for $n \geq 3$.
We deduce the following corollaries:

\begin{cor}
Let $\K$ be a quadratically closed field. Then $\NT_n(\K)$ is, up to similarity, the sole
maximal linear subspace of $\Mat_n(\K)$ with a trivial spectrum.
\end{cor}

\begin{cor}
Let $\K$ be a finite field with at least three elements.
Let $V$ be a maximal linear subspace of $\Mat_n(\K)$ with a trivial spectrum. \\
Then there are matrices $M_1,\dots,M_p$, either equal to $0 \in \Mat_1(\K)$ or belonging to $\Mat_2(\K)$ with no
eigenvalue in $\K$, such that
$$V \simeq \K\,M_1 \vee \cdots \vee \K\,M_p.$$
Each $M_k$ is then uniquely determined by $V$ up to similarity and multiplication by a non-zero scalar.
\end{cor}

We may finally state the structure theorem for affine subspaces of non-singular matrices.

\begin{theo}[Classification theorem for large affine subspaces of non-singular matrices]\label{classaffine}
Assume that $\# \K \geq 3$. Let $\calV$ be a $\binom{n}{2}$-dimensional affine subspace of $\Mat_n(\K)$
included in $\GL_n(\K)$. Then there is a list
$(P_1,\dots,P_p)\in \GL_{n_1}(\K) \times \cdots \times \GL_{n_p}(\K)$ of non-isotropic matrices such that
$n=n_1+\cdots+n_p$ and
$$\calV \sim I_n+\bigl(P_1\Mata_{n_1}(\K)\vee \cdots \vee P_p\Mata_{n_p}(\K)\bigr).$$
The integer $p$ is uniquely determined by $\calV$ and, for $1 \leq k \leq p$,
the similarity class of the non-isotropic quadratic form $X \mapsto X^TP_kX$
is uniquely determined by $\calV$.
Moreover, given another list $(Q_1,\dots,Q_p) \in \GL_{n_1}(\K) \times \cdots \times \GL_{n_p}(\K)$,
if $X \mapsto X^TQ_kX$ is similar to $X \mapsto X^TP_kX$ for each $k \in \lcro 1,p\rcro$, then
$$\calV \sim I_n+\bigl(Q_1\Mata_{n_1}(\K)\vee \cdots \vee Q_p\Mata_{n_p}(\K)\bigr).$$
\end{theo}

Note that the existence of $(P_1,\dots,P_p)$ is a trivial consequence of Theorem \ref{classtrivialspectrum} using the considerations
of Paragraph \ref{pureintro}.

As a consequence, $\binom{n}{2}$-dimensional affine subspaces of $\Mat_n(\K)$ included in $\GL_n(\K)$
are classified, up to equivalence, by the lists of the form $([\varphi_1],\dots,[\varphi_p])$
where the $\varphi_k$'s are finite-dimensional non-isotropic quadratic forms over $\K$, the $[\varphi_k]$'s are their similarity classes, and
$\underset{k=1}{\overset{p}{\sum}} \dim \varphi_k=n$.
For the field of real numbers, this has the following striking corollary:

\begin{cor}
Let $\calV$ be an affine subspace of $\Mat_n(\R)$
included in $\GL_n(\R)$ with dimension $\binom{n}{2}$.
Then there is a unique list $(n_1,\dots,n_p)$ of positive integers such that $n=n_1+\cdots+n_p$ and
$$\calV \sim I_n+\bigl(\Mata_{n_1}(\R) \vee \cdots \vee \Mata_{n_p}(\R)\bigr).$$
\end{cor}

\subsection{Totally intransitive action of a space of matrices}

Proving the previous theorems will require an extensive use of the following concept
and of the subsequent remark:

\begin{Def}
Let $V$ be a linear subspace of $\Mat_n(\K)$.
For $X \in \K^n$, set
$$VX:=\bigl\{MX \mid X \in V\bigr\}.$$
Note that $VX$ is always a linear subspace of $\K^n$. \\
We say that $V$ acts \textbf{totally intransitively} on $\K^n$ if
$VX \neq \K^n$ for every $X \in \K^n$, which is equivalent to having
$\dim(VX)<n$ for every $X \in \K^n$.
\end{Def}

\begin{Rem}
If $V$ has a trivial spectrum, then $X \not\in VX$ for every $X \in \K^n \setminus \{0\}$, hence $V$ acts totally intransitively on $\K^n$. \\
Moreover $V^T:=\bigl\{M^T \mid M \in V\bigr\}$ also has a trivial spectrum, hence
$$\forall X \in \K^n, \quad \dim(VX)<n \quad \text{and} \quad \dim(V^TX)<n.$$
\end{Rem}

\subsection{Structure of the paper}

We will start (Section \ref{espacesPA}) with general considerations on the spaces of the type $P\Mata_n(\K)$ with $P \in \GL_n(\K)$.
Using some of the obtained results, we will then prove the uniqueness statements
in Theorems \ref{classtrivialspectrum} and \ref{classaffine} (Section \ref{sectionunicite}).
The proof of Theorem \ref{irreducible} will be carried out in Section \ref{classificationsection} by induction on $n$, starting
from $n=2$ and using a recent lemma that was proved in \cite{dSPlargerank}: this is, by far, the most technical part of the paper.
In Section \ref{nilpotent}, we will easily derive Gerstenhaber's theorem from Theorem \ref{classtrivialspectrum} in the case $\# \K \geq 3$.
In Section \ref{F2}, we will show that Theorem \ref{irreducible} fails for $n=3$ and $\K \simeq \F_2$.
The case $\# \K=2$ remains a very exciting challenge that we will not undertake here.

\section{Basic properties of the spaces $P\Mata_n(\K)$}\label{espacesPA}

We consider first $P\Mata_n(\K)$ for an arbitrary $P \in \GL_n(\K)$.
To start with, note that, for every $Q \in \GL_n(\K)$, one has
$$P\Mata_n(\K)\,Q=P(Q^T)^{-1} Q^T\Mata_n(\K)\,Q=\bigl(P(Q^T)^{-1}\bigr)\Mata_n(\K)$$
which immediately shows that $\{P\Mata_n(\K) \mid P \in \GL_n(\K)\}$
is an equivalence class (for the equivalence of spaces of matrices).

In order to move forward, we need some basic properties of $\Mata_n(\K)$:
for this, we equip $\K^n$ with the non-degenerate symmetric bilinear form
$(X,Y) \mapsto X^TY$.

\begin{lemme}\label{action1}
For any $X \in \K^n \setminus \{0\}$, one has
$$\Mata_n(\K)X=\{X\}^\bot$$
and in particular $\dim(\Mata_n(\K)X)=n-1$.
\end{lemme}

\begin{proof}
This is obvious if $X$ is the first vector $e_1$ of the canonical basis of $\K^n$.
In the general case, we may find some $P \in \GL_n(\K)$ such that $Pe_1=X$, and note that
\begin{multline*}
\Mata_n(\K)X=(P^T)^{-1}P^T\Mata_n(\K)Pe_1 =(P^T)^{-1} \Mata_n(\K)e_1 \\
=(P^T)^{-1} \{e_1\}^\bot=\{Pe_1\}^\bot=\{X\}^\bot.
\end{multline*}
\end{proof}

We may now determine, amongst the spaces of the above form, those with a trivial spectrum:

\begin{lemme}\label{anisotropicequivalence}
Let $P \in \GL_n(\K)$. Then $P\Mata_n(\K)$ has a trivial spectrum if and only if
$P$ is non-isotropic.
\end{lemme}

\begin{proof}
The ``if" part has already been dealt with in the beginning of Section \ref{1.3}. Assume that
$P$ is isotropic. Then obviously $(P^T)^{-1}$ is also isotropic, hence we find
a non-zero vector $X \in \K^n$ such that $X^T(P^T)^{-1}X=0$, i.e.\ $P^{-1}X \in \{X\}^\bot$.
Then Lemma \ref{action1} shows that $P^{-1}X=AX$ for some $A \in \Mata_n(\K)$ hence
$(PA)X=X$, which shows that $P\Mata_n(\K)$ does not have a trivial spectrum.
\end{proof}

\begin{prop}
Let $P \in \GL_n(\K)$ be a non-isotropic matrix. Then $P\Mata_n(\K)$ is an irreducible maximal subspace
with a trivial spectrum.
\end{prop}

\begin{proof}
It only remains to show that $P\Mata_n(\K)$ is irreducible.
We use a \emph{reductio ad absurdum} by assuming that it has a non-trivial stable subspace $F \subset \K^n$ with dimension $p \in \lcro 1,n-1\rcro$.
Then $F^\bot$ is stabilized by $(P\Mata_n(\K))^T=\Mata_n(\K)P^T$.
Choosing an arbitrary non-zero vector $X \in F$, we have $\dim(P\Mata_n(\K)X)=\dim\{X\}^\bot=n-1$ hence $p=n-1$. \\
However, choosing a non-zero vector $Y \in F^\bot$ yields $\dim(\Mata_n(\K)P^TY)=n-1$ hence $n-p=n-1$.
This yields $n=2$ and $p=1$, in which case every matrix of $P\Mata_n(\K)$ must be nilpotent
(since it has an eigenvector and $0$ is the sole possible eigenvalue in $\K$), contradicting the fact that
every non-zero matrix of $P\Mata_2(\K)$ is non-singular.
\end{proof}

We now investigate when two spaces of the form $P\Mata_n(\K)$ are similar. Here is our basic result:

\begin{lemme}\label{centralizer}
Let $P \in \GL_n(\K)$. Then $P\Mata_n(\K)=\Mata_n(\K)$ if and only if $P$ is a scalar multiple of the identity.
\end{lemme}

\begin{proof}
The ``if" part is trivial. Assume conversely that $P\Mata_n(\K)=\Mata_n(\K)$.
Let $X \in \K^n \setminus \{0\}$. Then $P\Mata_n(\K)X=\Mata_n(\K)X$ yields that $P$ stabilizes the hyperplane
$\{X\}^\bot$, hence $P^T$ stabilizes $\Vect(X)$. Since this holds for every non-zero $X \in \K^n$,
this shows that $P^T$ is a scalar multiple of the identity, hence $P$ also is.
\end{proof}

The following corollary will be our starting point for the uniqueness statement in Theorem \ref{classtrivialspectrum}:

\begin{prop}\label{basicuniqueness}
Let $(P,Q)\in \GL_n(\K)^2$. Then $P\Mata_n(\K) \simeq Q\Mata_n(\K)$
if and only if $P \approx \lambda Q$ for some $\lambda \in \K \setminus \{0\}$.
\end{prop}

\begin{proof}
If $P=\lambda\,RQR^T$ for some $R \in \GL_n(\K)$ and some $\lambda \in \K \setminus \{0\}$, then
$$P\Mata_n(\K)=RQR^T\Mata_n(\K)=R(Q R^T\Mata_n(\K)R)R^{-1}
=R(Q\Mata_n(\K))R^{-1}.$$
Conversely, assume that $P\Mata_n(\K)=R(Q\Mata_n(\K))R^{-1}$ for some $R\in \GL_n(\K)$.
Then the above computation yields $(RQR^T)^{-1}P\Mata_n(\K)=\Mata_n(\K)$ hence Lemma \ref{centralizer} yields
a non-zero scalar $\lambda$ such that $(RQR^T)^{-1}P=\lambda I_n$. Therefore $P=R(\lambda\,Q)R^T$.
\end{proof}

\begin{Rem}[A crucial remark]\label{lowertriangremark}
Let $E$ be a finite dimensional vector space and $b$ a (possibly non-symmetric) bilinear form on $E$
such that $\forall x \in E \setminus \{0\}, \; b(x,x)\neq 0$.
Given a non-zero vector $x \in E$, the hyperplane $H:=\{y \in E : \; b(x,y)=0\}$
is then a complementary subspace of $\Vect(x)$ in $E$. By induction on the dimension of spaces, it follows that
there exists a basis $(f_1,\dots,f_n)$ of $E$ which is \emph{right-orthogonal} for $b$, i.e.\
$b(f_i,f_j)=0$ for every $(i,j)\in \lcro 1,n\rcro^2$ satisfying $i<j$. \\
For a non-isotropic matrix $P \in \GL_n(\K)$, this may be interpreted as follows:
$P$ is congruent to a lower-triangular matrix $T$, and hence $P\Mata_n(\K)$ is similar to $T\Mata_n(\K)$.
This remark will play a major part in our proof of Theorem \ref{irreducible}.
\end{Rem}

Now, given non-isotropic matrices $P$ and $Q$ of $\GL_n(\K)$, we may examine when
the two affine subspaces $I_n+P\Mata_n(\K)$ and $I_n+Q \Mata_n(\K)$ are equivalent.

\begin{prop}\label{lemmeequivalence}
Let $P$ and $Q$ be non-isotropic matrices of $\GL_n(\K)$. \\
Then $I_n+P\Mata_n(\K) \sim I_n+Q\Mata_n(\K)$ if and only if the quadratic forms $X \mapsto X^TPX$ and $X \mapsto X^TQX$ are similar.
\end{prop}

\begin{proof}
\begin{itemize}
\item Assume first that
$I_n+P\Mata_n(\K) \sim I_n+Q\Mata_n(\K)$, and choose a pair $(R,S)\in \GL_n(\K)^2$ such that
$R(I_n+P\Mata_n(\K))=(I_n+Q\Mata_n(\K))S$.
Obviously $S$ belongs to $(I_n+Q\Mata_n(\K))S$, hence $S=R(I_n+PA)$ for some $A \in \Mata_n(\K)$.
By comparing the translation vector spaces of $R(I_n+P\Mata_n(\K))$ and $(I_n+Q\Mata_n(\K))S$,
we also find that $RP\Mata_n(\K)=Q\Mata_n(\K)S=Q(S^T)^{-1}\Mata_n(\K)$. Therefore Proposition \ref{basicuniqueness} yields
a non-zero scalar $\lambda$ such that $RP=\lambda\, Q(S^T)^{-1}$.
It follows that $S^T=(I_n-AP^T)R^T$ and
$$\lambda\, Q = RPS^T = RP(I_n-AP^T)R^T = RPR^T-(RP)A(RP)^T.$$
Since $A$ is alternate, we find that $\lambda\, X^TQX=X^T(RPR^T)X=(R^TX)^TP(R^TX)$ for every $X \in \K^n$, and
the quadratic forms $X \mapsto X^TQX$ and $X \mapsto X^TPX$ are similar because $R^T$ is non-singular.

\item Conversely, assume that $X \mapsto X^TQX$ and $X \mapsto X^TPX$ are similar. Then there is a non-singular matrix $R \in \GL_n(\K)$, a non-zero
scalar $\lambda$ and an alternate matrix $A'$ such that $\lambda Q=RPR^T+A'$.
The matrix $A:=-(RP)^{-1}A'((RP)^T)^{-1}$ is congruent to $-A'$ and is therefore alternate.
We set $S:=R(I_n+PA)$.
Note that $S=RP(P^{-1}+A)$ is non-singular: indeed, $\forall X\in \K^n \setminus \{0\}, \; X^T(P^{-1}+A)X=X^TP^{-1}X \neq 0$
since $P^{-1}$ is non-isotropic, hence $P^{-1}+A$ is non-singular. \\
However $S^T=(I_n-AP^T)R^T$, therefore
$$RPS^T=RPR^T-(RP)A(RP)^T=RPR^T+A'=\lambda\, Q.$$
We deduce that
$$R\bigl(P\Mata_n(\K)\bigr)=\lambda\, Q(S^T)^{-1} \Mata_n(\K)=\bigl(Q \Mata_n(\K)\bigr)S.$$
We have just proven that the affine subspaces $R(I_n+P\Mata_n(\K))$ and $(I_n+Q\Mata_n(\K))S$ have $S$ as common point
and have the same translation vector space, hence they are equal. This yields $I_n+P\Mata_n(\K) \sim I_n+Q\Mata_n(\K)$.
\end{itemize}
\end{proof}

\noindent Finally, the following lemma will be a major key to unlock our proof of Theorem \ref{irreducible}:

\begin{lemme}\label{dernierrangement}
Let $n \geq 3$. Assume $\# \K \geq 3$.
Let $V$ be a $\binom{n}{2}$-dimensional linear subspace of $\Mat_n(\K)$ which acts totally intransitively on $\K^n$. \\
Assume that there exists a linear hyperplane $H$ of $V$ such that $H \subset \Mata_n(\K)$. Then $V=\Mata_n(\K)$.
\end{lemme}

\begin{proof}
Let $A \in V$. We prove that $A$ is alternate, i.e.\ that
the quadratic form $q : X \mapsto X^TAX$ is zero. We denote by
 $(e_1,\dots,e_n)$ the canonical basis of $\K^n$. \\
Let $X \in \K^n \setminus \{0\}$. If $\dim(HX)=n-1$ then
$AX \in HX$ since $HX \subset VX \subsetneq \K^n$, and hence $q(X)=0$. \\
If $\dim(HX)=n-1$ for every $X \in \K^n \setminus \{0\}$, then we readily have $q=0$. \\
Assume now that $\dim(HX_1)<n-1$ for some $X_1 \in \K^n \setminus \{0\}$. \\
This shows that there exists $X_2 \in \K^n \setminus \Vect(X_1)$ such that $X_2^TMX_1=0$ for every
$M \in H$. Let $X_3 \in \K^n \setminus \Vect(X_1,X_2)$. We may choose a
non-singular matrix $P \in \GL_n(\K)$ such that $Pe_i=X_i$ for every $i \in \lcro 1,3\rcro$. \\
Then $V':=P^TVP$ acts totally intransitively on $\K^n$ and contains the hyperplane
$H':=P^THP \subset \Mata_n(\K)$. We now have $e_2^TMe_1=0$ for every $M \in H'$, hence
$H'$ is included in the space $V_1$ of all alternate matrices $A=(a_{i,j})$ of $\Mat_n(\K)$ such that
$a_{2,1}=0$. The dimension of this space is obviously $\binom{n}{2}-1$, and therefore
$H'=V_1$. Then it is obvious that $\dim (H'e_3)=n-1$ and hence $\dim(HX_3)=n-1$. \\
We have therefore proven that
$$\forall X \in \K^n \setminus \Vect(X_1,X_2), \; q(X)=0.$$
It now suffices to show that $q$ vanishes everywhere on $\Vect(X_1,X_2)$. \\
Let $X \in \Vect(X_1,X_2) \setminus \{0\}$.
We choose an arbitrary vector $X_3 \in \K^n \setminus \Vect(X_1,X_2)$.
The plane $\Vect(X,X_3)$ satisfies $\Vect(X,X_3) \cap \Vect(X_1,X_2)=\Vect(X)$.
Since $\card \K>2$, this plane has at least four distinct $1$-dimensional subspaces,
three of which are different from $\Vect(X)$. We deduce that the quadratic form $q$ vanishes
on at least three $1$-dimensional subspaces of $\Vect(X,X_3)$.
Classically, this shows that $q$ vanishes everywhere on $\Vect(X,X_3)$
(indeed, a non-zero homogeneous polynomial of degree $2$ on $\K^2$ has at most $2$ zeroes
in the projective line $\mathbb{P}(\K^2)$).
In particular $q(X)=0$. We deduce that $q=0$, which completes our proof.
\end{proof}

\section{The uniqueness statement in the two classification theorems}\label{sectionunicite}

The uniqueness statement in Theorem \ref{classtrivialspectrum} is equivalent to the following result, which we prove right away:

\begin{prop}\label{uniquenessproposition1}
Let $(P_1,\dots,P_p)$ and $(Q_1,\dots,Q_q)$ be two families of non-isotropic matrices, respectively
of $\GL_{n_1}(\K) \times \cdots \times \GL_{n_p}(\K)$ and $\GL_{m_1}(\K) \times \cdots \times \GL_{m_q}(\K)$.
In order that
$$P_1\Mata_{n_1}(\K) \vee \cdots \vee
P_p\Mata_{n_p}(\K) \simeq Q_1\Mata_{m_1}(\K) \vee \cdots \vee
Q_q\Mata_{m_q}(\K),$$
it is necessary and sufficient that $q=p$ and $P_k$ be congruent to a scalar multiple of $Q_k$ for every $k \in \lcro 1,p\rcro$.
\end{prop}

\begin{proof}
The ``sufficient condition" statement follows immediately from Proposition \ref{basicuniqueness}. \\
For the converse statement, set $V:=P_1\Mata_{n_1}(\K) \vee \cdots \vee
P_p\Mata_{n_p}(\K)$ and $W:= Q_1\Mata_{m_1}(\K) \vee \cdots \vee
Q_q\Mata_{m_q}(\K)$. \\
For $k \in \lcro 1,p\rcro$, set $F_k:=\K^{n_1+\dots+n_k} \times \{0\} \subset \K^n$
, where $n:=n_1+\cdots+n_p$.
Set also $F_0=\{0\}$ and denote by $(e_1,\dots,e_n)$ the canonical basis of $\K^n$.
Set $k \in \lcro 1,p\rcro$. Our key statement is the set of equalities:
$$\forall X \in F_k \setminus F_{k-1}, \quad \dim (VX)=n_1+\dots+n_k-1.$$
Note first that the case $X=e_{n_1+\cdots+n_{k-1}+1}$ follows trivially from Lemma \ref{action1}. \\
Consider now an arbitrary vector $X \in F_k \setminus F_{k-1}$. Then $e_1,\dots,e_{n_1+\cdots+n_{k-1}},X$
are linearly independent, and may therefore be completed as a basis $(e_1,\dots,e_{n_1+\cdots+n_{k-1}},X,f_2,\dots,f_{n_k})$
of $F_k$. Therefore
$$\calB:=(e_1,\dots,e_{n_1+\cdots+n_{k-1}},X,f_2,\dots,f_{n_k},e_{n_1+\cdots+n_k+1},\dots,e_n)$$
is a basis of $\K^n$ and the matrix of coordinates $R$ of $\calB$ in the canonical basis of $\K^n$
belongs to $\GL_{n_1}(\K) \vee \cdots \vee \GL_{n_p}(\K)$ and satisfies $Re_{n_1+\cdots+n_{k-1}+1}=X$.
Proposition \ref{basicuniqueness} thus yields a list  of non-isotropic matrices
$(P'_1,\dots,P'_p) \in \GL_{n_1}(\K) \times \cdots \times \GL_{n_p}(\K)$
for which
$$RVR^{-1} \subset P'_1 \Mata_{n_1}(\K) \vee \cdots \vee
P'_p\Mata_{n_p}(\K)$$
and therefore $RVR^{-1}=P'_1 \Mata_{n_1}(\K) \vee \cdots \vee
P'_p\Mata_{n_p}(\K)$ as the dimensions equal $\binom{n}{2}$ on both sides.
Applying the special case of $e_{n_1+\cdots+n_{k-1}+1}$ to $RVR^{-1}$ then yields $\dim (VX)=\dim(RVX)=\dim (RVR^{-1})(RX)=n_1+\cdots+n_k-1$.
\vskip 2mm
\noindent It follows that
$$\bigl\{\dim(VX)\mid X \in \K^n\bigr\}=\{0,n_1-1,n_1+n_2-1,\dots,n_1+\dots+n_p-1\}$$
has cardinality $p+1$. The same holds for $W$ instead of $V$ with the $m_j$'s in place of the $n_k$'s.
Since $V$ is similar to $W$, one has $\bigl\{\dim(VX)\mid X \in \K^n\bigr\}=\bigl\{\dim(WX)\mid X \in \K^n\bigr\}$
and we deduce successively that $q=p$ and $(n_1,\dots,n_p)=(m_1,\dots,m_q)$. Now, set $P \in \GL_n(\K)$ such that
$W=P^{-1}VP$. For every $k \in \lcro 1,p\rcro$, remark that
$$\bigl\{X \in \K^n: \; \dim VX \leq n_1+\dots+n_k-1\bigr\}=F_k=\bigl\{X \in \K^n:\; \dim WX \leq n_1+\dots+n_k-1\bigr\},$$
hence $P$ stabilizes $F_k$. This shows that $P \in \GL_{n_1}(\K)\vee \cdots \vee \GL_{n_p}(\K)$,
which in turn proves that $P_k\Mata_{n_k}(\K)$ is similar to $Q_k\Mata_{n_k}(\K)$ for every $k \in \lcro 1,p\rcro$.
Proposition \ref{basicuniqueness} finally yields that $P_k$ is congruent to a scalar multiple of $Q_k$, for every $k \in \lcro 1,p\rcro$.
\end{proof}

\begin{prop}\label{uniquenessproposition2}
Let $(P_1,\dots,P_p)$ and $(Q_1,\dots,Q_q)$ be two families of non-isotropic matrices, respectively
in $\GL_{n_1}(\K) \times \cdots \times \GL_{n_p}(\K)$ and $\GL_{m_1}(\K) \times \cdots \times \GL_{m_q}(\K)$.
In order that
$$\bigl(I_{n_1}+P_1\Mata_{n_1}(\K)\bigr) \vee \cdots \vee
\bigl(I_{n_p}+P_p\Mata_{n_p}(\K)\bigr) \sim \bigl(I_{m_1}+Q_1\Mata_{m_1}(\K)\bigr) \vee \cdots \vee
\bigl(I_{m_q}+Q_q\Mata_{m_q}(\K)\bigr),$$
it is necessary and sufficient that $q=p$ and that the (non-isotropic) quadratic form $X \mapsto X^TP_k X$
be similar to $X \mapsto X^TQ_kX$ for every $k \in \lcro 1,p\rcro$.
\end{prop}

\begin{proof}
The ``sufficient condition" statement follows trivially from Proposition \ref{lemmeequivalence}.
For the converse statement, let us set $\calV:=(I_{n_1}+P_1 \Mata_{n_1}(\K)) \vee \cdots \vee (I_{n_p}+P_p \Mata_{n_p}(\K))$
and $\calW:=(I_{m_1}+Q_1 \Mata_{m_1}(\K)) \vee \cdots \vee (I_{m_q}+Q_q \Mata_{m_q}(\K))$,
and assume that $\calV \sim \calW$. Choose two non-singular matrices $R$ and $S$ such that $\calW=R\calV S$.
Denote by $V$ (resp.\ by $W$) the translation vector space of $\calV$ (resp.\ of $\calW$), and set $n:=\underset{k=1}{\overset{p}{\sum}}n_k$.
Then
$$\calW=(RS)S^{-1}(I_n+V)S=(RS)(I_n+S^{-1}VS).$$
In particular $RS \in \calW$ and the comparison of translation vector spaces yields
$S^{-1}VS=(RS)^{-1}W$. The first result yields that $RS$ is upper block-triangular with diagonal blocks $R_1,\dots,R_q$ where
$R_k \in \GL_{m_k}(\K)$ for every $k \in \lcro 1,q\rcro$.
Thus
$$S^{-1}VS=(RS)^{-1}W=(R_1^{-1}Q_1) \Mata_{m_1}(\K) \vee \cdots \vee (R_q^{-1}Q_q) \Mata_{m_q}(\K)$$
and the $R_k^{-1}Q_k$'s are necessarily non-isotropic since $S^{-1}VS$ has a trivial spectrum.
We deduce from Proposition \ref{uniquenessproposition1} that $(n_1,\dots,n_p)=(m_1,\dots,m_q)$.
With the line of reasoning from the proof of Proposition \ref{uniquenessproposition1},
we also find that $S \in \GL_{n_1}(\K)\vee \cdots \vee \GL_{n_p}(\K)$. However we already know that
$RS$ belongs to $\GL_{n_1}(\K)\vee \cdots \vee \GL_{n_p}(\K)$ and hence $R=(RS)S^{-1} \in \GL_{n_1}(\K)\vee \cdots \vee \GL_{n_p}(\K)$.

Returning to $R\calV S=\calW$ finally entails that $I_{n_k}+Q_k\Mata_{n_k}(\K)$ is equivalent to
$I_{n_k}+P_k\Mata_{n_k}(\K)$ for each $k \in \lcro 1,p\rcro$, and Proposition \ref{lemmeequivalence}
then yields that $X \mapsto X^TP_kX$ is similar to $X \mapsto X^TQ_kX$ for each $k \in \lcro 1,p\rcro$.
\end{proof}

\section{Structure of the irreducible maximal spaces with a trivial spectrum}\label{classificationsection}

In the whole section, we assume $\# \K \geq 3$. We will prove Theorem \ref{irreducible} by induction.
The case $n=1$ needs no explanation.

\subsection{The case $n=2$}

Let $V$ be an irreducible maximal linear subspace of $\Mat_2(\K)$ with a trivial spectrum.
Then $V=\Vect(M)$ for some $M \in \Mat_2(\K) \setminus \{0\}$ with no non-zero eigenvalue.
If $0$ is an eigenvalue of $M$, then $M$ is triangularizable and $V$ is not irreducible. \\
Hence $M$ is non-singular.
Setting $K:=\begin{bmatrix}
0 & 1 \\
-1 & 0
\end{bmatrix}$ and $P:=MK^{-1}$, we readily have $P\Mata_2(\K)=\Vect(M)=V$ and Lemma \ref{anisotropicequivalence}
shows that $P$ is non-isotropic.

\subsection{Setting things up}\label{setupsection}

Let $n \geq 2$ and assume that
the result of Theorem \ref{irreducible} holds for any positive integer $k \leq n$.
Let $V \subset \Mat_{n+1}(\K)$ be a maximal subspace with a trivial spectrum.
Denote by $(e_1,\dots,e_{n+1})$ the canonical basis of $\K^{n+1}$.
We wish to show that $V$ is reducible or similar to $P\Mata_{n+1}(\K)$ for some $P \in \GL_{n+1}(\K)$,
in which case Lemma \ref{anisotropicequivalence} guarantees that $P$ must be non-isotropic.

Of course, this amounts to finding a basis of $\K^{n+1}$ in which all the endomorphisms $X \mapsto MX$ of $\K^{n+1}$, for $M \in V$,
have a ``reduced" shape that is essentially the one described in Theorem \ref{classtrivialspectrum}.
The first problem is how to select the last vector $f_{n+1}$ of such a basis.
Since the rank of an alternate matrix is even, an obvious necessary condition is that
$V$ should not contain any matrix with $\Vect(f_{n+1})$ as column space. Our starting point is that
such a vector exists (and may even be chosen amongst the canonical basis  of $\K^{n+1}$).
This has already been proven in \cite[Proposition 10]{dSPlargerank}: we reproduce a proof since it is short
and the result is crucial to our study.

\begin{lemme}
Let $W$ be a linear subspace of $\Mat_p(\K)$ with a trivial spectrum.
Then there exists a non-zero vector $X \in \K^p$ such that
$W$ contains no matrix $M$ with $\Vect(X)$ as column space.
\end{lemme}

\begin{proof}
Denote by $(e_1,\dots,e_p)$ the canonical basis of $\K^p$. For $X \in \K^p\setminus \{0\}$,
set $W_X:=\bigl\{M \in W : \; \im(M) \subset \Vect(X)\bigr\}$.
For $(i,j)\in \lcro 1,p\rcro^2$,
denote by $E_{i,j}$ the matrix of $\Mat_p(\K)$ with zero entries everywhere except at the $(i,j)$-spot
where the entry is $1$. \\
We prove, by induction on $p$, that there exists an index $i \in \lcro 1,p\rcro$ such that
$W_{e_i}=\{0\}$. The case $p=1$ is trivial. \\
Assume that $W_{e_i}\neq \{0\}$ for every $i \in \lcro 1,p\rcro$,
denote by $W'$ the linear subspace of $W$ consisting of its matrices with zero as last row, and write every $M \in W'$
as
$$M=\begin{bmatrix}
J(M) & [?]_{(p-1) \times 1} \\
[0]_{1 \times (p-1)} & 0
\end{bmatrix} \quad \text{with $J(M) \in \Mat_{p-1}(\K)$.}$$
Then $J(W')$ is a linear subspace of $\Mat_{p-1}(\K)$ with a trivial spectrum.
The induction hypothesis yields an index $i \in \lcro 1,p-1\rcro$ such that $J(W')_{e_i}=\{0\}$. \\
Since $W_{e_i}\neq \{0\}$, we find a matrix $M \in W$ such that $\im(M)=\Vect(e_i)$.
Then $M \in W'$ and it follows from $J(W')_{e_i}=\{0\}$ that $M$ is a non-zero scalar multiple of $E_{i,p}$.
Therefore $E_{i,p} \in W$. \\
Now, taking an arbitrary permutation matrix $P \in \GL_n(\K)$ and applying the previous step to $PWP^{-1}$ yields the following generalization:
for every $j \in \lcro 1,p\rcro$, there exists an integer $f(j)\in \lcro 1,p\rcro \setminus \{j\}$ such that $E_{f(j),j} \in W$. \\
We choose a \emph{cycle} for the map $f : \lcro 1,p\rcro \rightarrow \lcro 1,p\rcro$,
i.e.\ a list $(j_1,\dots,j_r)$ of distinct elements of $\lcro 1,p\rcro$ such that $f(j_1)=j_2,\dots,f(j_{r-1})=j_r$
and $f(j_r)=j_1$. The matrix $A:=\underset{k=1}{\overset{r}{\sum}}E_{f(j_k),j_k}$ then belongs to
$W$ although $1$ is an eigenvalue of it (a corresponding eigenvector being $\underset{k=1}{\overset{r}{\sum}} e_{j_k}$).
This is a contradiction, which shows that $W_{e_i}= \{0\}$ for some $i \in \lcro 1,p\rcro$.
\end{proof}

By conjugating $V$ with an appropriate invertible matrix,
we then lose no generality assuming that no matrix of $V$ has $\Vect(e_{n+1})$ as column space and that $Ve_{n+1} \subset \Vect(e_1,\dots,e_n)$
(since $e_{n+1} \not\in Ve_{n+1}$). This means that every matrix of $V$ has a $0$ entry at the $(n+1,n+1)$-spot.

\vskip 2mm
In order to complete the choice of a ``good" basis for $V$, we now turn to the first $n$ vectors $f_1,\dots,f_n$.
The basic idea is to find the projections of $f_1,\dots,f_n$ onto $\Vect(e_1,\dots,e_n)$ and alongside $\Vect(e_{n+1})$
by applying the induction hypothesis to a subspace of $\Mat_n(\K)$ that is deduced from $V$ (the space $V_{ul}$ defined below),
and then apply the induction hypothesis once more to find the projections of $f_1,\dots,f_n$ onto $\Vect(e_{n+1})$ alongside
$\Vect(e_1,\dots,e_n)$.

\vskip 2mm
Consider the subspace $W$ of $V$ consisting of its matrices with zero as last column.
For $M \in W$, write
$$M=\begin{bmatrix}
K(M) & [0]_{n \times 1} \\
L(M) & 0
\end{bmatrix} \quad \text{with $K(M) \in \Mat_n(\K)$ and $L(M) \in \Mat_{1,n}(\K)$,}$$
and set
$$V_{ul}:=K(W)$$
(the subscript ``$ul$" stands for ``upper left").
The rank theorem shows that
$$\dim V=\dim W+\dim (Ve_{n+1}) \quad \text{and} \quad  \dim W=\dim \Ker K+\dim V_{ul}.$$
However, our assumptions mean that $\Ker K=\{0\}$, hence
$$\dim V=\dim V_{ul}+\dim (Ve_{n+1}).$$
Obviously, $V_{ul}$ is a linear subspace of $\Mat_n(\K)$ with a trivial spectrum hence $\dim V_{ul} \leq \binom{n}{2}$.
Moreover $\dim (Ve_{n+1}) \leq n$ since $V$ acts totally intransitively on $\K^{n+1}$.
We deduce that
$$\binom{n+1}{2}=\dim V = \dim V_{ul}+\dim (Ve_{n+1}) \leq \binom{n}{2}+n=\binom{n+1}{2},$$
hence
$$\dim V_{ul}=\binom{n}{2} \quad \text{and} \quad \dim (Ve_{n+1})=n.$$
In this reduced situation, we conclude that:
\begin{enumerate}[1.]
\item $V_{ul}$ is a maximal linear subspace of $\Mat_n(\K)$ with a trivial spectrum.
\item $Ve_{n+1}=\Vect(e_1,\dots,e_n)$.
\end{enumerate}

\vskip 3mm
Applying the induction hypothesis to $V_{ul}$ together with Remark \ref{lowertriangremark} shows that
we may find non-isotropic \emph{lower-triangular} matrices $P_1,\dots,P_r$ such that
$$V_{ul} \simeq P_1\Mata_{n_1}(\K) \vee \cdots \vee P_r\Mata_{n_r}(\K).$$
This shows that, by conjugating $V$ with a well-chosen matrix of the form $\begin{bmatrix}
R & [0]_{n \times 1} \\
[0]_{1 \times n} & 1
\end{bmatrix}$ for some $R \in \GL_n(\K)$, we lose no generality assuming that
$$V_{ul}=P_1\Mata_{n_1}(\K) \vee \cdots \vee P_r\Mata_{n_r}(\K) \quad \text{and} \quad
P_1=\begin{bmatrix}
1 & [0]_{1 \times (n_1-1)} \\
C'_1 & P'_1
\end{bmatrix}$$
for some lower-triangular matrix $P'_1 \in \Mat_{n_1-1}(\K)$ (possibly of size $0$)
and some column matrix $C'_1 \in \Mat_{n_1-1,1}(\K)$.

\begin{Rem}[An important remark on block-diagrams]\label{blockrem}
\emph{From now on, and unless specified otherwise, every matrix $M$ of $V$ will be systematically seen with the following $3 \times 3$ block decomposition:}
$$M=\begin{bmatrix}
? & [?]_{1 \times (n-1)} & ? \\
[?]_{(n-1) \times 1} & [?]_{n-1} & [?]_{(n-1) \times 1} \\
? & [?]_{1 \times (n-1)} & ?
\end{bmatrix}$$
\emph{i.e.\ the four question marks represent single entries, whilst the others represent submatrices with sizes as indicated by the subscript
(where the central subscript $n-1$ denotes a $(n-1)\times (n-1)$ block).}
\end{Rem}
\vskip 2mm
\noindent If $n_1>1$, we set $s:=r$, $(i_1,\dots,i_s):=(n_1-1,n_2,\dots,n_r)$ and $(R_1,\dots,R_s):=(P'_1,P_2,\dots,P_r)$. \\
If $n_1=1$, we set $s:=r-1$, $(i_1,\dots,i_s):=(n_2,\dots,n_r)$ and $(R_1,\dots,R_s):=(P_2,\dots,P_r)$. \\
In any case, we set
$$V_m:=R_1\Mata_{i_1}(\K) \vee \cdots  \vee R_s\Mata_{i_s}(\K)$$
(the subscript ``$m$" stands for ``middle").
\noindent Here are two consequences of the above reductions (with the block decompositions laid out in Remark \ref{blockrem}):
\begin{enumerate}[(i)]
\item For every $L \in \Mat_{1,n-1}(\K)$, the subspace $V$ contains a matrix of the form
$$\begin{bmatrix}
? & L & 0 \\
? & ? & 0 \\
? & ? & 0
\end{bmatrix};$$
\item For every $U \in V_m$, the subspace $V$ contains a matrix of the form
$$\begin{bmatrix}
0 & 0 & 0 \\
0 & U & 0 \\
? & ? & 0
\end{bmatrix}.$$
\end{enumerate}

\begin{proof}[Proof of statement (i)]
Let $L_1 \in \Mat_{1,n_1-1}(\K)$. Then
$$P_1 \times \begin{bmatrix}
0 & L_1 \\
-L_1^T & [0]_{n_1-1}
\end{bmatrix}=\begin{bmatrix}
? & L_1 \\
[?]_{(n_1-1) \times 1} & [?]_{n_1-1}
\end{bmatrix}$$
and $\begin{bmatrix}
0 & L_1 \\
-L_1^T & [0]_{n_1-1}
\end{bmatrix}$ is alternate.
Since $V_{ul}=P_1\Mata_{n_1}(\K) \vee \cdots \vee P_r\Mata_{n_r}(\K)$, we deduce that,
for every $L \in \Mat_{1,n-1}(\K)$, the subspace $V_{ul}$ contains a matrix of the form
$\begin{bmatrix}
? & L  \\
[?]_{(n-1) \times 1} & [?]_{n-1}  \\
\end{bmatrix}$, and the conclusion follows from the definition of $V_{ul}$.
\end{proof}

\begin{proof}[Proof of statement (ii)]
We will only tackle the case $n_1>1$, the case $n_1=1$ being essentially similar (and even simpler).
For every $M \in P_2\Mata_{n_2}(\K) \vee \cdots \vee P_r\Mata_{n_r}(\K)$
and every $N \in \Mat_{n_1-1,n-n_1}(\K)$,
we know that $V_{ul}$ contains the matrix $\begin{bmatrix}
0 & [0]_{1 \times (n_1-1)} & [0]_{1 \times (n-n_1)} \\
[0]_{(n_1-1) \times 1} & [0]_{n_1-1} & N  \\
[0]_{(n-n_1)\times 1} & [0]_{(n-n_1)\times (n_1-1)} & M \\
\end{bmatrix}$.
Let $A \in \Mata_{n_1-1}(\K)$. Then
$$P_1 \times \begin{bmatrix}
0 & [0]_{1 \times (n_1-1)} \\
[0]_{(n_1-1) \times 1} & A
\end{bmatrix}=\begin{bmatrix}
0 & [0]_{1 \times (n_1-1)} \\
[0]_{(n_1-1) \times 1} & P'_1 A
\end{bmatrix}$$
and it follows that $V_{ul}$ contains a matrix of the form
$$\begin{bmatrix}
0 & [0]_{1 \times (n_1-1)} & [0]_{1 \times (n-n_1)} \\
[0]_{(n_1-1) \times 1} & P'_1 A & [0]_{(n_1-1) \times (n-n_1)} \\
[0]_{(n-n_1) \times 1} & [0]_{(n-n_1) \times (n_1-1)} & [0]_{(n-n_1) \times (n-n_1)}
\end{bmatrix}.$$
With the respective definitions of $V_m$ and $V_{ul}$, point (ii) follows easily.
\end{proof}

\vskip 2mm
Let now $C \in \Mat_{n-1,1}(\K)$.
Since $Ve_{n+1}=\Vect(e_1,\dots,e_n)$, we know that $V$ contains a matrix of the form
$$\begin{bmatrix}
? & ? & 0 \\
? & ? & C \\
? & ? & 0
\end{bmatrix}.$$
Adding an appropriate matrix given by statement (i), and remembering that $0$ is the only possible eigenvalue for a matrix in
$V$, we deduce:
\begin{itemize}
\item[(iii)] $V$ contains a matrix of the form
$$\begin{bmatrix}
0 & 0 & 0 \\
? & ? & C \\
? & ? & 0
\end{bmatrix}.$$
\end{itemize}
Denote now by $V'$ the subspace of $V$ consisting of its matrices with zero as first row.
For $M \in V'$, write
$$M=\begin{bmatrix}
0 & [0]_{1\times n} \\
[?]_{n \times 1}  & \calJ(M)
\end{bmatrix} \quad \text{with $\calJ(M) \in \Mat_n(\K)$,}$$
and set
$$V_{lr}:=\calJ(V')$$
(the subscript ``$lr$" stands for ``lower right").
Note that the subspace $V_{lr}$ of $\Mat_n(\K)$ has a trivial spectrum and that it
contains:
\begin{itemize}
\item[(a)] A matrix of the form $\begin{bmatrix}
U & [0]_{(n-1) \times 1} \\
[?]_{1 \times (n-1)} & 0
\end{bmatrix}$ for every $U \in V_m$ (by statement (ii));
\item[(b)] A matrix of the form
$\begin{bmatrix}
[?]_{n-1} & C \\
[?]_{1 \times (n-1)} & 0
\end{bmatrix}$ for every $C \in \Mat_{n-1,1}(\K)$ (by statement (iii)).
\end{itemize}
Since $\dim V_m=\binom{n-1}{2}$, we deduce that $\dim V_{lr} \geq \binom{n-1}{2}+(n-1)=\binom{n}{2}$.
However $\dim V_{lr} \leq \binom{n}{2}$ since $V_{lr}$ has a trivial spectrum.
It thus follows from statements (a) and (b) that:
\begin{itemize}
\item[(c)] $V_{lr}$ contains, for every $U \in V_m$, a \emph{unique} matrix of the form $\begin{bmatrix}
U & [0]_{(n-1) \times 1} \\
? & 0
\end{bmatrix}$;
\item[(d)] Every matrix of $V_{lr}$ with zero as last column has the form $\begin{bmatrix}
U & [0]_{(n-1) \times 1} \\
? & 0
\end{bmatrix}$ for some $U \in V_m$.
\end{itemize}

\vskip 2mm
A key point now is that $V_{lr}$ is a maximal linear subspace of $\Mat_n(\K)$ with a trivial spectrum.
One may thus be tempted to apply the induction hypothesis to $V_{lr}$.
However, the problem is that using a new change of basis blindingly
risks destroying the previous reduced form of $V_{ul}$! As we shall now see,
the fact that $V_m$ is already reduced forces $V_{lr}$ to be already in the reduced form of Theorem \ref{classtrivialspectrum}
(i.e.\ no further change of basis is necessary at this point).

\begin{claim}\label{compatclaim}
The subspace $V_{lr}$ has a ``roughly reduced" shape i.e.\ there exists an integer $q \geq 1$, a non-isotropic matrix $Q \in \GL_q(\K)$
and a maximal subspace $W$ of $\Mat_{n-q}(\K)$ with a trivial spectrum such that
$$V_{lr}=W \vee Q\Mata_q(\K).$$
\end{claim}

\begin{proof}
Applying the induction hypothesis to $V_{lr}$, we recover a matrix
$P\in \GL_n(\K)$, a non-isotropic matrix $Q' \in \GL_q(\K)$ (possibly with $q=n$) and a maximal subspace $W'$ of 
$\Mat_{n-q}(\K)$ with a trivial spectrum such that
$$PV_{lr}P^{-1}=W'\vee Q' \Mata_q(\K).$$
Note, using statement (b), that $\dim(V_{lr}e_n)=n-1$ whereas $\dim(PV_{lr}P^{-1}x)<n-1$ for every $x \in \Vect(e_1,\dots,e_{n-q})$
(since $W'$ acts totally intransitively on $\K^{n-q}$).
Hence $Pe_n \not\in \Vect(e_1,\dots,e_{n-q})$. Multiplying $P$ with a well-chosen matrix of
$\GL_{n-q}(\K)\vee \GL_q(\K)$, we lose no generality assuming that $Pe_n=e_n$.

\vskip 2mm
Assume first that $q=1$. Then $V_{lr} e_n=\Vect(e_1,\dots,e_{n-1})=(PV_{lr}P^{-1})e_n$ whilst
$PV_{lr}P^{-1}e_n=P(V_{lr}e_n)$, which shows that $P$ stabilizes $\Vect(e_1,\dots,e_{n-1})$.
Therefore $P \in \GL_{n-1}(\K) \vee \{1\}$ and
$V_{lr}=W \vee \Mata_1(\K)$ for some maximal linear subspace $W$ of $\Mat_{n-1}(\K)$ with a trivial spectrum.

\vskip 2mm
Assume, for the rest of the proof, that $q>1$. Our aim is to prove that $P \in \GL_{n-q}(\K) \vee \GL_q(\K)$, and it will follow that
$V_{lr}=W \vee Q\Mata_q(\K)$ for some maximal linear subspace $W$ of $\Mat_{n-q}(\K)$
with a trivial spectrum and some non-isotropic matrix $Q \in \GL_q(\K)$. \\
Set $$H:=\bigl\{M \in V_{lr} : \; Me_n=0\bigr\}$$
i.e.\ $H$ is the set of all matrices of
$V_{lr}$ with $0$ as last column.
Notice that $PHP^{-1}=\bigl\{M \in PV_{lr}P^{-1} : \; Me_n=0\bigr\}$ since $Pe_n=e_n$.
Notice also that
$$\Vect(e_1,\dots,e_{n-1-i_s}) \subset \Vect(e_1,\dots,e_{n-1})=V_{lr}e_n$$
(this uses statement (b) and the fact that
$V_{lr} e_n \neq \K^n$)
and that
$$\Vect(e_1,\dots,e_{n-q}) \subset (PV_{lr}P^{-1})e_n\; .$$

\begin{itemize}
\item \emph{Case 1: $i_s>1$.}
\begin{itemize}
\item We first claim that
\begin{equation}\label{importantdimclaim1}
\forall x \in V_{lr}e_n, \quad \dim Hx<n-2 \; \Leftrightarrow \; x \in \Vect(e_1,\dots,e_{n-1-i_s}).
\end{equation}
Indeed, let $x \in \Vect(e_1,\dots,e_{n-1})$ seen as a vector of $\K^{n-1}$ with the
canonical identification $\K^{n-1} \simeq \K^{n-1} \times \{0\} \subset \K^n$.
By statements (c) and (d), one has
$$\dim V_m x \leq \dim Hx \leq 1+\dim V_m x.$$
If $x \in \Vect(e_1,\dots,e_{n-1-i_s})$, then the line of reasoning from the proof of
Proposition \ref{uniquenessproposition1} yields $\dim V_m x \leq n-i_s-2$ and hence $\dim Hx \leq n-i_s-1<n-2$;
otherwise $\dim V_m x=n-2$ and hence $\dim Hx \geq n-2$. \\
\item Moreover, we claim that
\begin{equation}\label{importantdimclaim2}
\forall x \in (PV_{lr}P^{-1})e_n, \; \dim (PHP^{-1}x)<n-2 \; \Leftrightarrow \; x \in \Vect(e_1,\dots,e_{n-q}).
\end{equation}
The implication $\Leftarrow$ follows from $PV_{lr}P^{-1}=W'\vee Q' \Mata_q(\K)$
since $W'$ acts totally intransitively on $\K^{n-q}$ and $q>1$. \\
For the converse implication, notice first that the equality $PV_{lr}P^{-1}=W' \vee Q'\Mata_q(\K)$ yields
$(PV_{lr}P^{-1})e_n=\Vect(e_1,\dots,e_{n-q})\oplus G$ for some $(q-1)$-dimensional subspace $G$ of
$\Vect(e_{n-q+1},\dots,e_n)$ which does not contain $e_n$
(note that $(PV_{lr}P^{-1})e_n$ cannot contain $e_n$ since $PV_{lr}P^{-1}$ has a trivial spectrum).
Consider a vector $x \in G \setminus \{0\}$.
The subspace $PHP^{-1}$ contains, for every $A \in \Mata_{q-1}(\K)$, and every $B \in \Mat_{n-q,q}(\K)$ with zero as last column,
the matrix
$$\begin{bmatrix}
[0]_{n-q} & B \\
[0]_{q \times (n-q)} & C
\end{bmatrix} \quad \text{where
$C=Q' \times \begin{bmatrix}
A & [0]_{(q-1) \times 1} \\
[0]_{1 \times (q-1)} & 0
\end{bmatrix}$.}$$
Since $x$ belongs to $\Vect(e_{n-q+1},\dots,e_n)$ and is linearly independent from
$e_n$, it easily follows that $\dim (PHP^{-1})x \geq  n-2$. \\
Let now $x \in (PV_{lr}P^{-1})e_n \setminus \Vect(e_1,\dots,e_{n-q})$.
Then we have a decomposition $x=z+y$ with $z \in \Vect(e_1,\dots,e_{n-q})$ and $y\in G \setminus \{0\}$. \\
Obviously, there exists a non-singular matrix $R \in \{I_{n-q}\} \vee \{I_q\}$ such that $Rx=y$.
Replacing $P$ with $RP$, we thus reduce the situation to the one where $x \in G \setminus \{0\}$,
which we have treated before. Implication $\Rightarrow$ in statement \eqref{importantdimclaim2}
follows.
\end{itemize}
Since $X \mapsto PX$ is linear, $Pe_n=e_n$,  $\Vect(e_1,\dots,e_{n-q}) \subset (PV_{lr}P^{-1})e_n$,
and $\Vect(e_1,\dots,e_{n-1-i_s}) \subset V_{lr} e_n$, we deduce from statements \eqref{importantdimclaim1}
and \eqref{importantdimclaim2} that $X \mapsto PX$ induces an isomorphism from
$\Vect(e_1,\dots,e_{n-1-i_s})$ to $\Vect(e_1,\dots,e_{n-q})$, hence $i_s=q-1$ and $P \in \GL_{n-q}(\K) \vee \GL_q(\K)$.

\vskip 2mm
\item \emph{Case 2: $i_s=1$.}
\begin{itemize}
\item Notice first that $\Vect(e_1,\dots,e_{n-1})=V_{lr}e_n$ and
\begin{equation}\label{importantdimclaim3}
\forall x \in V_{lr}e_n, \; \dim(Hx \cap V_{lr}e_n)<n-2 \quad \text{if $x \in \Vect(e_1,\dots,e_{n-2})$.}
\end{equation}
Indeed, for every $x \in  \Vect(e_1,\dots,e_{n-2})$, statements (c) and (d) show that
$\dim(Hx \cap V_{lr}e_n) \leq \dim(V_m x)$ (where $x$ is naturally seen as a vector of $\K^{n-1}$),
and the definition of $V_m$ shows, since $i_s=1$, that $\dim (V_m x)<n-2$.
\item On the other hand, we claim that
\begin{equation}\label{importantdimclaim4}
\forall x \in (PV_{lr}P^{-1})e_n, \; \dim\bigl((PHP^{-1})x \cap (PV_{lr}P^{-1})e_n\bigr)<n-2 \; \Leftrightarrow \; x \in \Vect(e_1,\dots,e_{n-q}).
\end{equation}
Indeed, for any $x \in \Vect(e_1,\dots,e_{n-q})$, one has
$$\dim\bigl((PHP^{-1})x \cap (PV_{lr}P^{-1})e_n\bigr) \leq \dim (PV_{lr}P^{-1})x \leq n-q-1<n-2.$$
Conversely, let $x \in (PV_{lr}P^{-1})e_n \setminus \Vect(e_1,\dots,e_{n-q})$.
Note first that $(PHP^{-1})x \subset (PV_{lr}P^{-1})e_n$. In order to see this, we naturally identify
$\K^n$ with $\K^{n-q} \oplus \K^q$: the identity $PV_{lr}P^{-1}=W' \vee Q'\Mata_q(\K)$ yields
$(PV_{lr}P^{-1})e_n=\K^{n-q} \times \bigl[Q'(\K^{q-1} \times \{0\})\bigr]$ whilst,
for every $M \in PHP^{-1}$, the column space of $M$ is included in $\K^{n-q} \times Q'\im(N)$,
where $N=\begin{bmatrix}
A_M & [0]_{(q-1) \times 1} \\
[0]_{1 \times (q-1)} & 0
\end{bmatrix}$ for some $A_M \in \Mata_{q-1}(\K)$;
this shows that $\im M \subset (PV_{lr}P^{-1})e_n$ for every $M \in PHP^{-1}$. \\
With the same arguments as in the proof of statement \eqref{importantdimclaim2}, one may prove that
$\dim (PHP^{-1})x=n-2$, and hence $\dim \bigl((PHP^{-1})x \cap (PV_{lr}P^{-1})e_n\bigr)=\dim (PHP^{-1})x=n-2$.
Therefore statement \eqref{importantdimclaim4} is established.
\end{itemize}
From statements \eqref{importantdimclaim3} and \eqref{importantdimclaim4},
we deduce that the linear injection $X \mapsto PX$ maps $\Vect(e_1,\dots,e_{n-2})$ into $\Vect(e_1,\dots,e_{n-q})$, which shows that
$q=2$, $i_s=1=q-1$ and $P \in \GL_{n-q}(\K) \vee \GL_q(\K)$. This finishes our proof.
\end{itemize}
\end{proof}

Now that we know that $V_{lr}$ is ``roughly reduced", we may use the shape of $V_m$ to better grasp the one of $V_{lr}$.

\vskip 2mm
Take $W$, $q$ and $Q$ as in Claim \ref{compatclaim}. If $q=1$, then obviously $W=V_m$. \\
Assume now that $q>1$ and split $Q=\begin{bmatrix}
Q_1 & [?]_{(q-1) \times 1} \\
[?]_{1 \times (q-1)} & ?
\end{bmatrix}$ with $Q_1 \in \Mat_{q-1}(\K)$. Then $Q_1$ is still non-isotropic
and statement (d) shows that $V_m$ contains $W \vee Q_1\Mata_{q-1}(\K)$, and hence
$V_m=W \vee Q_1\Mata_{q-1}(\K)$ since the dimensions are equal on both sides.
By applying the induction hypothesis to $W$
and by using the same arguments as in the proof of Proposition \ref{uniquenessproposition1},
we deduce that $W=R_1\Mata_{i_1}(\K) \vee \cdots \vee R_{s-1}\Mata_{i_{s-1}}(\K)$ and $Q_1\Mata_{q-1}(\K)=R_s\Mata_{i_s}(\K)$. \\
Therefore
$$V_{lr}=\begin{cases}
R_1\Mata_{i_1}(\K) \vee \cdots  \vee R_s\Mata_{i_s}(\K) \vee  \Mata_1(\K) \quad & \text{if $q=1$} \\
R_1\Mata_{i_1}(\K) \vee \cdots  \vee R_{s-1}\Mata_{i_{s-1}}(\K) \vee Q\Mata_q(\K) \quad&  \text{if $q>1$}.
\end{cases}$$
Assume again that $q>1$. Then $Q$ need not be lower-triangular, so we have to reduce the situation a little further. \\
Since $Ve_{n+1}=\Vect(e_1,\dots,e_n)$, we find that
$Q\Mata_q(\K)e_q=\Vect(e_1,\dots,e_{q-1})$ which shows that $Q$ stabilizes $\Vect(e_1,\dots,e_{q-1})$, i.e.\
$Q=\begin{bmatrix}
T_0 & [?]_{(q-1)\times 1} \\
[0]_{1 \times (q-1)} & \alpha
\end{bmatrix}$ for some $T_0 \in \GL_{q-1}(\K)$ and some $\alpha \in \K \setminus \{0\}$. \\
Note, since $Q$ is non-singular, that a matrix of the form $M=QA$, with $A \in \Mata_q(\K)$, has zero as last column if
and only if $A$ has zero as last column. It then follows from the shape of $V_m$ that
$T_0\Mata_{q-1}(\K)=R_s \Mata_{i_s}(\K)$. Therefore $(R_s^{-1}T_0)\Mata_{q-1}(\K)=\Mata_{q-1}(\K)$, and we deduce from
Lemma \ref{centralizer} that $T_0$ is a scalar multiple of $R_s$.
Since we may replace $Q$ with a scalar multiple of itself, we lose no generality assuming that
$T_0=R_s$. \\
Finally we define
$T_1:=\begin{bmatrix}
I_{q-1} & C' \\
[0]_{1 \times (q-1)} & 1
\end{bmatrix} \in \GL_q(\K)$, where $C':=\dfrac{C}{\alpha}$, so that
$Q':=(T_1)^TQT_1$ is lower-triangular; we replace
$V$ with $RVR^{-1}$ for
$R:=\begin{bmatrix}
I_{n+1-q} & [0]_{(n+1-q) \times q} \\
[0]_{q \times (n+1-q)} & T_1^T
\end{bmatrix}$.

\vskip 3mm
\noindent For the sake of convenience (and symmetry), we now set
$$P:=P_1 \quad \text{and} \quad p:=n_1.$$
Let us see how the situation looks like after all those reductions:
\begin{enumerate}[(i)]
\item We still have
$$V_{ul}=P\Mata_p(\K) \vee P_2\Mata_{n_2}(\K) \vee \cdots \vee P_r\Mata_{n_r}(\K)$$
and
$$V_m=R_1\Mata_{i_1}(\K) \vee \cdots \vee R_s\Mata_{i_s}(\K)$$
with the above notations (nothing has changed there). \\
Recall that $(i_1,\dots,i_s)=(n_2,\dots,n_r)$ if $p=1$, otherwise
$(i_1,\dots,i_s)=(n_1-1,n_2,\dots,n_r)$.

\item Either $q=1$ and then
$$V_{lr}=R_1\Mata_{i_1}(\K) \vee \cdots \vee R_s\Mata_{i_s}(\K)\vee Q\Mata_1(\K) \quad \text{with $Q=1$,}$$
or $q>1$ and then
$$V_{lr}=R_1\Mata_{i_1}(\K) \vee \cdots \vee R_{s-1}\Mata_{i_{s-1}}(\K)\vee Q\Mata_q(\K)$$
and
$$Q=\begin{bmatrix}
R_s & [0]_{i_s \times 1} \\
L_1 & \alpha
\end{bmatrix} \quad \text{with $\alpha \in \K \setminus \{0\}$ and $L_1 \in \Mat_{1,q-1}(\K)$}.$$
We set $\alpha:=1$ if $q=1$.
\item Recall finally that if $p>1$, then
$P=\begin{bmatrix}
1 & [0]_{1 \times (p-1)} \\
C_1 & R_1
\end{bmatrix}$ for some $C_1 \in \Mat_{p-1,1}(\K)$.

\item No matrix of $V$ has $\Vect(e_{n+1})$ as column space (no change there).
\end{enumerate}
However, one important thing has changed: if $q>1$, we no longer have $Ve_{n+1}=\Vect(e_1,\dots,e_n)$,
rather $Ve_{n+1}=\Vect(e_1,\dots,e_{n+1-q})\oplus H$ for some linear hyperplane
$H$ of $\Vect(e_{n+2-q},\dots,e_{n+1})$ which does not contain $e_{n+1}$. We still have $e_1 \in Ve_{n+1}$, nevertheless.
Set finally
$$Z:=\begin{bmatrix}
R_1 & & (0) \\
& \ddots & \\
(0) & & R_s
\end{bmatrix} \in \GL_{n-1}(\K).$$
From there, $V$ will remain essentially fixed. We will prove separately:
\begin{itemize}
\item That the case
$p=n=q$ (i.e.\ $V_{ul}$ and $V_{lr}$ are \textbf{glued}) leads to the equivalence of $V$ with $\Mata_{n+1}(\K)$;
\item That the case $p \neq n$ or $q \neq n$ (i.e.\ $V_{ul}$ and $V_{lr}$ are \textbf{unglued}) leads to the reducibility of $V$.
\end{itemize}
Prior to studying the two cases separately, we continue with general considerations that apply to both of them.

\subsection{Special types of matrices in $V$}

With the matrices $L_1$ and $C_1$ from the previous paragraph\footnote{Setting $L_1:=0$ if $q=1$, and $C_1:=0$ if $p=1$.}, set
$$\widetilde{L_1}:=\begin{bmatrix}
[0]_{1 \times (n-q)} & L_1
\end{bmatrix}\in \Mat_{1,n-1}(\K) \quad
\text{and} \quad
\widetilde{C_1}:=\begin{bmatrix}
C_1 \\
[0]_{(n-p) \times 1}
\end{bmatrix}\in \Mat_{n-1,1}(\K).$$

\begin{Not}
For an arbitrary $L \in \Mat_{1,n-1}(\K)$, we define
$\overline{L}$ as the matrix of $\Mat_{1,n-1}(\K)$
with the same first $p-1$ entries as $L$ and all the other ones equal to zero. \\
For an arbitrary $C \in \Mat_{n-1,1}(\K)$, we define
$\overline{C}$ as the matrix of $\Mat_{n-1,1}(\K)$
with the same last $q-1$ entries as $C$ and all the other ones equal to zero.
\end{Not}

\noindent
Using the respective shapes of $V_{ul}$, $V_m$ and $V_{lr}$, we now find important classes of matrices in $V$, together with an isolated matrix.
First of all, taking arbitrary row matrices
$L_0 \in \Mat_{1,p-1}(\K)$ and $L'_0 \in \Mat_{1,n-p}(\K)$,
we know that $V_{ul}$ contains a matrix of the form
$\begin{bmatrix}
P\,A & N \\
[0]_{(n-p) \times p} & [0]_{n-p}
\end{bmatrix}$ with $A=\begin{bmatrix}
0 & L_0 \\
-L_0^T & [0]_{p-1}
\end{bmatrix}$ and $N=\begin{bmatrix}
L'_0 \\
[0]_{(p-1) \times (n-p)}
\end{bmatrix}$.
Therefore, using the block decomposition of matrices of $V$ explained in Remark \ref{blockrem}, we find that:
\begin{itemize}
\item For every $L \in \Mat_{1,n-1}(\K)$, there is a unique\footnote{As the map $K$ from the beginning of
Section \ref{setupsection} is one-to-one.} $A_L \in V$
of the form
$$A_L=\begin{bmatrix}
0 & L & 0 \\
-Z\overline{L}^T & \widetilde{C_1}\overline{L} & 0 \\
f(L) & \varphi(L) & 0
\end{bmatrix},$$
and $f : \Mat_{1,n-1}(\K)\rightarrow \K$ and $\varphi : \Mat_{1,n-1}(\K)\rightarrow \Mat_{1,n-1}(\K)$
are linear maps.
\end{itemize}

Let $U \in V_m$, which we write as a block-triangular matrix $U=\begin{bmatrix}
[?]_{n-1-i_s} & [?]_{n-1-i_s,i_s} \\
[0]_{i_s,n-1-i_s} & R_s A
\end{bmatrix}$ with $A \in \Mata_s(\K)$.
With the respective structures of $V_{ul}$ and $V_m$ and the fact that $V$ contains no matrix with column space $\Vect(e_{n+1})$,
we know that $V$ contains a unique matrix of the form
$\begin{bmatrix}
0 & 0 & 0 \\
0 & U & 0  \\
? & ? & 0
\end{bmatrix}$. Since $Q \times \begin{bmatrix}
A & [0]_{(i_s-1) \times 1} \\
[0]_{1 \times (i_s-1)} & 0
\end{bmatrix}=\begin{bmatrix}
R_s A & [0]_{(i_s-1) \times 1} \\
L_1R_s^{-1}(R_sA) & 0
\end{bmatrix}$, the structure of $V_{lr}$ yields that the above matrix of
$V$ has $\begin{bmatrix}
? & \widetilde{L_1}Z^{-1}U & 0
\end{bmatrix}$ as last row. Therefore:
\begin{itemize}
\item For every $U \in V_m$, there is a unique $E_U \in V$ of the form
$$E_U=\begin{bmatrix}
0 & 0 & 0 \\
0 & U & 0 \\
h(U) & \widetilde{L_1}Z^{-1}U & 0
\end{bmatrix}.$$
\end{itemize}
\noindent We know that some matrix of $V$ has $\begin{bmatrix}
1 & 0 & \cdots & 0
\end{bmatrix}^T$ as last column. Summing it with a well-chosen matrix of type $A_L$, we deduce:
\begin{itemize}
\item The subspace $V$ contains a matrix
$$J=\begin{bmatrix}
a & 0 & 1 \\
C'_1 & ? & 0 \\
b & L'_1 & 0
\end{bmatrix} \quad \text{with $(a,b)\in \K^2$ and $(L'_1,C'_1) \in \Mat_{1,n-1}(\K)\times \Mat_{n-1,1}(\K)$.}$$
\end{itemize}

With the above matrices $A_L$ and $J$, we find that $\dim(e_1^TV) \geq n$.
We already knew that $\dim V=\binom{n+1}{2}$ and $\dim V_{lr}=\binom{n}{2}$, hence the rank theorem shows that the map
$\calJ$ from Section \ref{setupsection} yields an isomorphism from the subspace of all matrices of
$V$ with zero as first row to $V_{lr}$. Using the structure of $V_{lr}$ with the same method
as in the definition of the $A_L$ matrices, we thus find one last important class of matrices in $V$:
\begin{itemize}
\item For every $C \in \Mat_{n-1,1}(\K)$, there is a unique $B_C \in V$
of the form
$$B_C=\begin{bmatrix}
0 & 0 & 0 \\
\psi(C) & 0 & C \\
g(C) & -\alpha \,\overline{C}^T (Z^{-1})^T & \widetilde{L_1}Z^{-1}C
\end{bmatrix}$$
and $g : \Mat_{n-1,1}(\K)\rightarrow \K$ and $\psi : \Mat_{n-1,1}(\K)\rightarrow \Mat_{n-1,1}(\K)$
are linear maps.
\end{itemize}

\begin{Rem}
The above matrices span $V$: a straightforward computation shows
indeed that the linear subspaces
$\bigl\{A_L \mid L \in \Mat_{1,n-1}(\K)\bigr\}$,
$\bigl\{B_C \mid C \in \Mat_{n-1,1}(\K)\bigr\}$,
$\bigl\{E_U \mid U \in V_m\bigr\}$ and $\Vect(J)$ are
independent, and the sum of their dimensions is
$(n-1)+(n-1)+\binom{n-1}{2}+1=\binom{n+1}{2}=\dim V$.
\end{Rem}

\vskip 2mm
From now on, our main task is to refine our understanding of the matrices of the types $A_L$, $B_C$, $E_U$ and $J$:
the basic strategy is to form well-chosen linear combinations of those special matrices and
use the fact that none of them may have a non-zero eigenvalue. Most of the time, we will simply apply the
fact that both $V$ and $V^T$ act totally intransitively on $\K^{n+1}$.
Let us start by considering the maps $\varphi$ and $\psi$ in the $A_L$ and $B_C$ matrices.

\begin{claim}
The maps $\varphi$ and $\psi$ are scalar multiples of the identity.
\end{claim}

\begin{proof}
Let $C \in \Mat_{n-1,1}(\K)$ and $L \in \Mat_{1,n-1}(\K)$. Denote by $x$ (resp.\ $y$) the vector of $\Vect(e_2,\dots,e_n)$
with coordinate matrix $C$ (resp.\ $L^T$) in the basis $(e_2,\dots,e_n)$.
We prove that
\begin{equation}\label{preserveorthogonality}
LC=0 \Rightarrow \bigl(\varphi(L)\,C=0 \; \text{and} \; L\,\psi(C)=0\bigr).
\end{equation}
Assume that $LC=0$. Notice then that both $A_L$ and
$B_C$ stabilize the plane $\Vect(x,e_{n+1})$ and that the respective matrices of
their induced endomorphisms in the basis $(x,e_{n+1})$ are
$\begin{bmatrix}
0 & 0 \\
\varphi(L)C & 0
\end{bmatrix}$ and $\begin{bmatrix}
0 & 1 \\
t_1 & t_2
\end{bmatrix}$ for some $(t_1,t_2)\in \K^2$.
Since $V$ has a trivial spectrum, we deduce that
$$\forall \lambda \in \K, \quad \begin{vmatrix}
1 & 1 \\
t_1+\lambda\, \varphi(L)C & 1+t_2
\end{vmatrix} \neq 0,$$
hence $\varphi(L)C=0$. \\
Similarly, notice that $A_L^T$ and $B_C^T$ both stabilize $\Vect(e_1,y)$ and the respective matrices of their induced
endomorphisms in the basis $(e_1,y)$ are
$\begin{bmatrix}
0 & s_1 \\
1 & s_2
\end{bmatrix}$ and
$\begin{bmatrix}
0 & L \psi(C) \\
0 & 0
\end{bmatrix}$ for some $(s_1,s_2)\in \K^2$. With the above line of reasoning, we deduce that $L\psi(C)=0$. \\
We may now conclude. For the non-degenerate bilinear mapping $(L,C) \mapsto LC$ on $\Mat_{1,n-1}(\K) \times \Mat_{n-1,1}(\K)$,
we deduce from \eqref{preserveorthogonality}  that $\varphi$ stabilizes the orthogonal subspace of every linear hyperplane of $\Mat_{1,n-1}(\K)$,
hence $\varphi$ stabilizes every $1$-dimensional linear subspace of $\Mat_{1,n-1}(\K)$, which shows that $\varphi$ is a scalar multiple of the identity.
With the same line of reasoning, we see that $\psi$ is also a scalar multiple of the identity.
\end{proof}

\noindent We now have two scalars $\lambda$ and $\mu$ such that:
$$\forall L \in \Mat_{1,n-1}(\K), \quad
A_L=\begin{bmatrix}
0 & L & 0 \\
-Z\overline{L}^T & \widetilde{C_1}\overline{L} & 0 \\
f(L) & \lambda\,L & 0
\end{bmatrix}$$
and
$$\forall C \in \Mat_{n-1,1}(\K), \quad
B_C=\begin{bmatrix}
0 & 0 & 0 \\
\mu\,C & 0 & C \\
g(C) & -\alpha \,\overline{C}^T (Z^{-1})^T & \widetilde{L_1}Z^{-1}C
\end{bmatrix}.$$

\begin{claim}\label{claimhu0}
The map $h$ vanishes everywhere on $V_m$.
\end{claim}

\begin{proof}
Choose $t \in \K$ such that $\mu+t \neq 0$ and $a+t \neq 0$ (this is feasible since $\# \K \geq 3$).
Remark then that
$$\begin{cases}
\forall U \in V_m, \quad & E_U(e_1+t\,e_{n+1})=\begin{bmatrix}0 \\
[0]_{(n-1) \times 1} \\
 h(U)
\end{bmatrix} \\
\forall C \in \Mat_{n-1,1}(\K), \quad & B_C(e_1+t\,e_{n+1})=\begin{bmatrix}
0 \\
 (\mu+t)\,C \\
?
\end{bmatrix} \\
& J(e_1+t\,e_{n+1})=\begin{bmatrix}
a+t \\
? \\
?
\end{bmatrix}.
\end{cases}$$
However $V(e_1+te_{n+1})$ is a strict linear subspace of $\K^{n+1}$.
Judging from the vectors $B_C(e_1+t\,e_{n+1})$ and the vector $J(e_1+t\,e_{n+1})$,
we deduce that $V(e_1+te_{n+1})$ cannot contain $e_{n+1}$. This
shows that $h(U)=0$ for every $U \in V_m$.
\end{proof}

\noindent It follows that
$$\forall U \in V_m, \quad E_U=\begin{bmatrix}
0 & 0 & 0 \\
0 & U & 0 \\
0 & \widetilde{L_1}Z^{-1}U & 0
\end{bmatrix}.$$
From there, we need to study the glued and unglued cases separately.

\subsection{The case $V_{ul}$ and $V_{lr}$ are glued}

In this section, we assume $p=q=n$. In this case, we simply have
$\widetilde{L_1}=L_1$, $\widetilde{C_1}=C_1$, $Z=R_1=R_s$, $V_m=Z\Mata_{n-1}(\K)$ and
$\forall (L,C) \in \Mat_{1,n-1}(\K) \times \Mat_{n-1,1}(\K), \; \overline{L}=L \; \text{and} \; \overline{C}=C$.
Our aim is to prove that $V$ is equivalent to $\Mata_{n+1}(\K)$.

\begin{claim}
One has
$$\forall (L,C)\in \Mat_{1,n-1}(\K) \times \Mat_{n-1,1}(\K), \quad f(L)=-L_1L^T \quad \text{and} \quad
g(C)=\mu L_1Z^{-1}C.$$
\end{claim}

\begin{proof}
Let $t \in \K \setminus \{-a\}$.
Note that $J(e_1+te_{n+1})$ has $a+t$ as first entry, whereas
$$\begin{cases}
\forall L \in \Mat_{1,n-1}(\K), \; & A_L(e_1+te_{n+1})=\begin{bmatrix}
0 \\
 -ZL^T \\
  f(L)
\end{bmatrix} \\
\forall C \in \Mat_{n-1,1}(\K), \; & B_C(e_1+te_{n+1})=\begin{bmatrix}
0 \\
 (\mu+t)C \\
  g(C)+tL_1Z^{-1}C
\end{bmatrix}.
\end{cases}$$
Judging from $J(e_1+te_{n+1})$, the vector space $V(e_1+te_{n+1})$ cannot contain $\Vect(e_2,\dots,e_{n+1})$.
Thus $V(e_1+te_{n+1}) \cap \Vect(e_2,\dots,e_{n+1})=\bigl\{A_L(e_1+te_{n+1}) \mid L \in \Mat_{1,n-1}(\K)\bigr\}$
(since the first space has a dimension lesser than $n$ and obviously contains the second one).
Using the $B_C$ matrices, it follows that
$$\forall C \in \Mat_{n-1,1}(\K), \quad g(C)+tL_1Z^{-1}C=(\mu+t)\,f\bigl(-C^T (Z^{-1})^T\bigr).$$
Since this holds for several values of $t$, we deduce that
$$\forall C \in \Mat_{n-1,1}(\K), \quad g(C)=-\mu f(C^T (Z^{-1})^T) \quad \text{and} \quad L_1Z^{-1}C=-f\bigl(C^T (Z^{-1})^T\bigr),$$
which obviously yields the claimed results.
\end{proof}

\noindent Therefore, for any $(L,C,U)\in \Mat_{1,n-1}(\K) \times \Mat_{n-1,1}(\K) \times Z\,\Mata_{n-1}(\K)$,
we have
$$A_L=\begin{bmatrix}
0 & L & 0 \\
-ZL^T & C_1 L & 0 \\
-L_1L^T & \lambda\,L & 0
\end{bmatrix} \quad ; \quad
B_C=\begin{bmatrix}
0 & 0 & 0 \\
\mu\,C & 0 & C \\
\mu\,L_1Z^{-1}C & -\alpha \,C^T (Z^{-1})^T & L_1 Z^{-1}C
\end{bmatrix}$$
and
$$E_U=\begin{bmatrix}
0 & 0 & 0 \\
0 & U & 0 \\
0 & L_1Z^{-1}U & 0
\end{bmatrix}.$$
Set now
$$T:=\begin{bmatrix}
1 & 0 & 0 \\
C_1 & Z & 0 \\
\lambda & L_1 & \alpha
\end{bmatrix}\in \GL_{n+1}(\K)
\quad \text{and} \quad T':=
\begin{bmatrix}
1 & 0 & 0 \\
0 & I_{n-1} & 0 \\
-\mu & 0 & 1
\end{bmatrix} \in \GL_{n+1}(\K).$$
A straightforward computation shows that, for every $(L,C,U)\in \Mat_{1,n-1}(\K) \times \Mat_{n-1,1}(\K) \times (Z\Mata_{n-1}(\K))$:
$$T^{-1}A_L\,T'=\begin{bmatrix}
0 & L & 0 \\
-L^T & 0 & 0 \\
0 & 0 & 0
\end{bmatrix} \quad ; \quad
T^{-1}B_C\,T'=\begin{bmatrix}
0 & 0 & 0 \\
0 & 0 & Z^{-1}C \\
0 & -(Z^{-1}C)^T & 0
\end{bmatrix}$$
and
$$T^{-1}E_U\, T'=\begin{bmatrix}
0 & 0 & 0 \\
0 & Z^{-1}U & 0 \\
0 & 0 & 0
\end{bmatrix}.$$
Therefore $T^{-1}VT'$ contains a linear hyperplane of $\Mata_{n+1}(\K)$.
Since $V$ acts totally intransitively on $\K^{n+1}$, this is also the case of $T^{-1}VT'$,
hence Lemma \ref{dernierrangement} shows that $T^{-1}VT'=\Mata_{n+1}(\K)$.
We deduce that $V$ is equivalent to $\Mata_{n+1}(\K)$ and may thus be written as $Y\Mata_{n+1}(\K)$ for some
$Y \in \GL_{n+1}(\K)$, and Lemma \ref{anisotropicequivalence} yields that $Y$ is non-isotropic.
This completes the case where $V_{ul}$ and $V_{lr}$ are glued.

\subsection{The case $V_{ul}$ and $V_{lr}$ are unglued}

Here, we assume that $p<n$ or $q<n$.
Note that this means that $p=1$ or $q=1$ or
there are several diagonal blocks $R_1 \Mata_{i_1}(\K),\dots,R_s \Mata_{i_s}(\K)$ in the block decomposition of $V_m$ discussed earlier.
Note in particular that $p+q \leq n+1$.

\vskip 2mm
Our aim is to prove that $V$ is reducible. Since the matrices $A_L$, $B_C$, $E_U$ and $J$ span $V$,
it suffices to find a non-trivial linear subspace of $\K^{n+1}$ which is stabilized by all of them.
In that prospect, we start by analyzing $f$ and $g$.

\begin{claim}\label{claimf=0}
One has $f=0$, and $g(C)=0$ for every $C \in \Mat_{n-1,1}(\K)$ such that $\overline{C}=0$.
\end{claim}

\begin{proof}
We start by proving that
$$\forall L \in \Mat_{1,n-1}(\K), \quad \overline{L}=0 \Rightarrow f(L)=0 \qquad
\text{and} \qquad \forall C \in \Mat_{n-1,1}(\K), \quad \overline{C}=0 \Rightarrow g(C)=0.$$
We choose $t \in \K$ such that $\mu+t \neq 0$ and $a+t \neq 0$.
Then, for every $L \in \Mat_{1,n-1}(\K)$ such that $\overline{L}=0$,
one has
$$A_L(e_1+te_{n+1})=\begin{bmatrix}
0 \\
 [0]_{(n-1) \times 1} \\
  f(L)
\end{bmatrix},$$
hence $f(L)=0$ with the same argument as in the proof of Claim \ref{claimhu0}. \\
Choose now $x \in \K$ such that $\lambda+x \neq 0$ and $x \neq 0$.
Then
$$\begin{cases}
\forall L \in \Mat_{1,n-1}(\K), \quad & (xe_1+e_{n+1})^TA_L=\begin{bmatrix}
f(L) & (\lambda+x)L & 0
\end{bmatrix} \\
& (xe_1+e_{n+1})^TJ=\begin{bmatrix}
? & [?]_{1 \times (n-1)} & x
\end{bmatrix}.
\end{cases}$$
Since $V^T(xe_1+e_{n+1})$ is a strict linear subspace of $\K^{n+1}$,
those matrices show that $e_1$ cannot belong to $V^T(xe_1+e_{n+1})$.
However
$$\forall C \in \Mat_{n-1,1}(\K), \; (xe_1+e_{n+1})^TB_C
=\begin{bmatrix}
g(C) & -\alpha\,\overline{C}^T(Z^{-1})^T & \widetilde{L_1}Z^{-1}C
\end{bmatrix}.$$
Therefore, if $\overline{C}=0$, then $\widetilde{L_1}Z^{-1}C=0$ and hence $g(C)=0$. \\
Let $L \in \Mat_{1,n-1}(\K)$.
The column matrix $C:=Z\overline{L}^T$ has null entries starting from the $p$-th, and
since $p+q\leq n+1$, this yields $\overline{C}=0$. Therefore $g(C)=0$ and
$$\bigl((\mu+t)A_L+B_C\bigr)(e_1+te_{n+1})=\begin{bmatrix}
0 \\
[0]_{(n-1) \times 1} \\
(\mu+t)f(L)
\end{bmatrix}.$$
The above argument then shows that $f(L)=0$.
\end{proof}

In particular, we have
$$\forall L \in \Mat_{1,n-1}(\K), \; A_L=\begin{bmatrix}
0 & L & 0 \\
-Z\overline{L}^T & \widetilde{C_1}\overline{L} & 0 \\
0 & \lambda\,L & 0
\end{bmatrix}.$$
We now distinguish between two cases, whether $p<n$ or $p=n$.

\begin{claim}\label{claimp<n}
If $p<n$, then $Ve_1 \subset \Vect(e_1,\dots,e_p)$.
\end{claim}

\begin{proof}
Assume that $p<n$. \\
Write $C'_1=\begin{bmatrix}
c'_1 \\
\vdots \\
c'_{n-1}
\end{bmatrix}$.
Let $i \in \lcro p,n-1\rcro$ (note that such an integer exists). \\
Let $(x,y,z)\in \K^3$ such that $x+\lambda z \neq 0$.
Denote by $C''_i \in \Mat_{n-1,1}(\K)$ the column matrix with all entries $0$
except the $i$-th which equals $1$.
Note that, for every $L \in \Mat_{1,n-1}(\K)$, both column matrices $\widetilde{C_1}$ and $Z\overline{L}^T$
have zero entries starting from the $p$-th: for $\widetilde{C_1}$, this comes from its very definition; for
$Z\overline{L}^T$, this is obvious if $p=1$ because then $\overline{L}=0$, otherwise this comes from the fact that
$Z$ stabilizes $\K^{p-1} \times \{0\} \subset \K^{n-1}$ (as $i_1=p-1$) and that $\overline{L} \in \K^{p-1} \times \{0\}$.
It follows that the $(i+1)$-th row of every $A_L$ matrix is zero. Setting $\gamma:=\widetilde{L_1}Z^{-1}C''_i$, we therefore have:
$$\begin{cases}
\forall L \in \Mat_{1,n-1}(\K), \; & (xe_1+ye_{i+1}+ze_{n+1})^TA_L=\begin{bmatrix}
0 & (x+\lambda z)\,L & 0
\end{bmatrix} \\
& (xe_1+ye_{i+1}+ze_{n+1})^TJ=\begin{bmatrix}
ax+c'_iy+bz & [?]_{1 \times (n-1)} & x
\end{bmatrix} \\
& (xe_1+ye_{i+1}+ze_{n+1})^T B_{C''_i}=
\begin{bmatrix}
\mu y+g(C''_i)z & [?]_{1 \times (n-1)} & y+\gamma z
\end{bmatrix}.
\end{cases}$$
Since $V^T(xe_1+ye_{i+1}+ze_{n+1}) \neq \K^{n+1}$, we deduce that
$$\begin{vmatrix}
ax+c'_iy+bz & x \\
\mu y+g(C''_i)z & y+\gamma z
\end{vmatrix}=0.$$
Notice that, with an arbitrary $(y,z)\in \K^2$ being fixed, the above equation is linear in $x$ and has several solutions, hence
$$(c'_iy+bz)(y+\gamma z)=0 \quad \text{and} \quad a(y+\gamma z)-(\mu y+g(C''_i)z)=0.$$
Both equations have a degree lesser than or equal to $2$ in both variables. Since $\card \K>2$, we deduce that
$$c'_i=0 \quad ; \quad c'_i\gamma+b=0 \quad ; \quad a=\mu \quad \text{and} \quad a\gamma=g(C''_i).$$
Therefore $a=\mu$, $b=0$ and $c'_p=\dots=c'_{n-1}=0$.
Since $Z$ is non-singular and stabilizes $\K^{p-1} \times \{0\}$, we may thus find $L \in \Mat_{1,n-1}(\K)$ such that $C'_1=Z\overline{L}^T$.
The first column of $A_L+J$ is $\begin{bmatrix}
a \\
 0 \\
 \vdots \\
  0
\end{bmatrix}$, therefore $a=0$ (because $A_L+J$ has no non-zero eigenvalue). It follows that
$\mu=0$ and $g(C''_i)=0$ for every $i \in \lcro p,n-1\rcro$.
Since $g$ is linear, $g(C)=0$ whenever $\overline{C}=0$ (by Claim \ref{claimf=0}),
and $p-1<n-q+1$, we deduce that $g=0$. \\
For any matrix of type $A_L$, $B_C$, $E_U$ or $J$, we have therefore found that
its first column has null entries starting from the $(p+1)$-th. This yields our claim
since these matrices span $V$.
\end{proof}

\begin{claim}\label{claimp=n}
Assume that $p=n$ (and therefore $q=1$). Then $\lambda=b=0$ and $L'_1=0$.
\end{claim}

\noindent This shows that all the matrices $A_L$, $B_C$, $E_U$ and $J$ have zero as last row in the case $p=n$.

\begin{proof}
Since $q=1$, one has $\widetilde{L_1}=0$,
whilst $\overline{C}=0$ for every $C \in \Mat_{1,n-1}(\K)$. This leads to $f=0$ and $g=0$ by Claim \ref{claimf=0}. \\
Therefore
$$\forall (L,C)\in \Mat_{1,n-1}(\K) \times \Mat_{n-1,1}(\K), \;
A_L=\begin{bmatrix}
0 & L & 0 \\
-Z L^T & ? & 0 \\
0 & \lambda\,L & 0
\end{bmatrix} \quad \text{and} \quad
B_C=\begin{bmatrix}
0 & 0 & 0 \\
\mu\,C & 0 & C \\
0 & 0 & 0
\end{bmatrix}.$$
Write $L'_1=\begin{bmatrix}
l'_1 & \cdots & l'_{n-1}
\end{bmatrix}$.
Let $i \in \lcro 1,n-1\rcro$.
Denote by $L''_i \in \Mat_{1,n-1}(\K)$ the row matrix with all entries zero except the $i$-th which equals one. \\
Let $(x,z)\in \K^2$ such that $\mu x+z \neq 0$.
Then
$$\begin{cases}
\forall C \in \Mat_{n-1,1}(\K), \; & B_C (x e_1 +e_{i+1}+ze_{n+1})=\begin{bmatrix}
0 \\
 (\mu x+z)C \\
  0
\end{bmatrix} \\
& A_{L''_i} (x e_1 +e_{i+1}+ze_{n+1})=
\begin{bmatrix}
1 \\
 [?]_{(n-1) \times 1} \\
  \lambda
\end{bmatrix} \\
& J (x e_1 +e_{i+1}+ze_{n+1})=
\begin{bmatrix}
ax+z \\
  [?]_{(n-1) \times 1} \\
  bx+l'_i
\end{bmatrix}.
\end{cases}$$
We deduce that
$\begin{vmatrix}
1 & ax+z \\
\lambda & bx+l'_i
\end{vmatrix}=0$. Since, for a given $x \in \K$, this holds for several values of $z$, we successively deduce that
$\lambda=0$ and $\forall x \in \K, \; b\,x+l'_i=0$, which yields $\lambda=b=l'_i=0$.
Therefore $L'_1=0$.
\end{proof}

In two special cases, we may now conclude that $V$ is reducible:
if $p=1$ then Claim \ref{claimp<n} shows that $\Vect(e_1)$ is stabilized by $V$;
if $p=n$, then Claims \ref{claimf=0} and \ref{claimp=n} show that $\Vect(e_1,\dots,e_n)$ is stabilized by $V$
(indeed, in that case $q=1$ and hence $\widetilde{L_1}=0$ and $\overline{C}=0$ for every $C \in \Mat_{n-1,1}(\K)$).

\vskip 3mm
Assume finally that $1<p<n$. Then $Ve_1 \subset \Vect(e_1,\dots,e_p)$ by Claim \ref{claimp<n}.
Note that the change of basis matrix
$R=\begin{bmatrix}
I_{n+1-q} & 0 \\
0 & T_1^T
\end{bmatrix}$ from Section \ref{setupsection} leaves $\Vect(e_1,\dots,e_p)$
invariant as $p \leq n+1-q$. Therefore we also have $(R^{-1}VR)e_1 \subset \Vect(e_1,\dots,e_p)$, and some of our recent findings
may be summed up as follows:

\begin{prop}\label{bilan}
Let $V$ be a maximal subspace of $\Mat_{n+1}(\K)$ with a trivial spectrum such that:
\begin{enumerate}[(i)]
\item $Ve_{n+1}=\Vect(e_1,\dots,e_n)$;
\item There are lower-triangular non-isotropic matrices $P \in \GL_p(\K)$,
$P_2 \in \GL_{n_2}(\K),\dots,
P_r \in \GL_{n_r}(\K)$, with $1<p<n$, such that
$V_{ul}=P \Mata_p(\K) \vee P_2 \Mata_{n_2}(\K) \vee \cdots \vee P_r \Mata_{n_r}(\K)$.
\end{enumerate}
Then $Ve_1 \subset \Vect(e_1,\dots,e_p)$.
\end{prop}

Note that the fact that $V$ contains no matrix with column space $\Vect(e_{n+1})$, our starting point in Section \ref{setupsection},
is a consequence of assumptions (i) and (ii) of Proposition \ref{bilan} (using the rank theorem to compute the dimension of $V$ from that of $V_{ul}$,
as in the beginning of Section \ref{setupsection}).

Now, all we need to complete the unglued case is to show that any $V$ satisfying
the assumptions of Proposition \ref{bilan} is reducible. Let $V$ be such a subspace,
with the above notations.
Let $x \in \Vect(e_1,\dots,e_p) \setminus \{0\}$.
Recall that the bilinear form $b : (X,Y) \in (\K^p)^2 \mapsto X^TPY$ is non-isotropic, and hence non-degenerate.
Denote by $X_0$ the matrix of coordinates of $x$ in $(e_1,\dots,e_p)$.
In the hyperplane $H:=\{Y \in \K^p : \; X_0^TY=0\}$, we may therefore find a
``right-sided orthogonal basis" $(f_2,\dots,f_p)$, i.e.\ $b(f_i,f_j)=0$ for every $(i,j)\in \lcro 2,p\rcro^2$ with $i<j$.
We then choose a non-zero vector $f_1$ such that $b(f_1,f_j)=0$ for every $j \in \lcro 2,p\rcro$.
It follows that $(f_1,\dots,f_p)$ is a basis of $\K^p$.
Denoting by $S$ the matrix of coordinates of $(f_1,f_2,\dots,f_p)$ in $(e_1,\dots,e_p)$,
the matrix $P':=S^TPS$ is lower-triangular and
$$S^T\bigl(P \Mata_p(\K)\bigr)(S^T)^{-1}=P' \Mata_p(\K).$$
Set then $T_2:=\begin{bmatrix}
S^T & 0  \\
0 & I_{n+1-p}
\end{bmatrix}\in \GL_{n+1}(\K)$ and
$$V':=T_2\,V\, T_2^{-1}.$$
Notice finally that $T_2$ stabilizes $\Vect(e_1,\dots,e_n)$, fixes $e_{n+1}$,
and obviously
$$V'_{ul}=P'\Mata_p(\K) \vee P_2 \Mata_{n_2}(\K) \vee \cdots \vee P_r \Mata_{n_r}(\K).$$
Thus Proposition \ref{bilan} applied to $V'$ shows that $V'e_1 \subset \Vect(e_1,\dots,e_p)$.
However $S$ maps $\Vect(e_2,\dots,e_p)$ to $\Vect(f_2,\dots,f_p)$, hence $S^TX_0 \in \Vect(e_1) \setminus \{0\}$. This yields
$$Vx \subset \Vect(e_1,\dots,e_p).$$
We conclude that $\Vect(e_1,\dots,e_p)$ is a non-trivial invariant subspace for $V$,
hence $V$ is reducible. This completes our proof of Theorem \ref{irreducible}.

\section{On large spaces of nilpotent matrices}\label{nilpotent}

In this short section, we show that the following famous theorem of Gerstenhaber on
linear subspaces of nilpotent matrices is an easy consequence of Theorem \ref{classtrivialspectrum}:

\begin{theo}[Gerstenhaber's theorem]
Let $\K$ be a field with at least three elements, and
$V$ be a linear subspace of $\Mat_n(\K)$ such that $\dim V=\binom{n}{2}$
and every matrix of $V$ is nilpotent. Then $V$ is similar to $\NT_n(\K)$.
\end{theo}

See \cite{Gerstenhaber} for the original proof under the more restrictive assumption $\# \K \geq n$,
\cite{Mathes} for a very elegant proof using trace maps and a theorem of Jacobson, and
\cite{Serezhkin} for a proof with no restriction on the cardinality of $\K$.

\begin{proof}
The assumptions show that $V$ is a maximal linear subspace of $\Mat_n(\K)$ with a trivial spectrum.
Then $V \simeq P_1\Mata_{n_1}(\K) \vee \cdots \vee P_p\Mata_{n_p}(\K)$ for non-isotropic matrices
$P_1,\dots,P_p$. Since every matrix of $V$ is nilpotent, every matrix of $P_k \Mata_{n_k}(\K)$
is nilpotent for every $k \in \lcro 1,p\rcro$. \\
Let $q \geq 2$ be a positive integer and $P \in \GL_q(\K)$, and assume that $P$ is non-isotropic and
every element of $P\Mata_q(\K)$ is nilpotent. Note that $q$ is odd since $\Mata_q(\K)$ contains
non-singular matrices when $q$ is even.
Then $\tr(PA)=0$ for every $A \in \Mata_q(\K)$, which shows
that $P$ is symmetric. Since $q$ is odd and $P$ is non-singular, $P$ is not alternate hence
it is congruent to a non-singular diagonal matrix $D$ (even if $\K$ has characteristic 2, see \cite[Chapter 35]{invitquad}).
Thus $D\Mata_q(\K)$ is similar to $P\Mata_q(\K)$ and must therefore have a trivial spectrum.
Finally, set $K:=\begin{bmatrix}
0 & 1 \\
-1 & 0
\end{bmatrix}$ and $A:=\begin{bmatrix}
K & [0]_{2 \times (q-2)} \\
[0]_{(q-2) \times 2} & [0]_{q-2}
\end{bmatrix} \in \Mata_q(\K)$, and note that $DA$ is obviously non-nilpotent, a contradiction. \\
Returning to $V$, we deduce that $n_1=\dots=n_p=1$ hence $V \simeq \NT_n(\K)$.
\end{proof}

\section{On the exceptional case of $\F_2$}\label{F2}

In the proof of Theorem \ref{classtrivialspectrum}, we have repeatedly used the
assumption that the field $\K$ had at least $3$ elements. The reader will therefore not be surprised by the following
counterexample which shows that Theorem \ref{classtrivialspectrum} fails for the field $\F_2$.
Remark first that there is no non-isotropic matrix in $\GL_3(\F_2)$ (since every $3$-dimensional quadratic form over a finite field
is isotropic), hence no maximal linear subspace of $\Mat_3(\F_2)$ with a trivial spectrum has the form $P\Mata_3(\F_2)$.

Consider the following matrices of $\Mat_3(\F_2)$:
$$A:=\begin{bmatrix}
0 & 1 & 0 \\
0 & 0 & 0 \\
0 & 1 & 0
\end{bmatrix} \quad ; \quad
B:=\begin{bmatrix}
1 & 0 & 1 \\
1 & 0 & 0 \\
1 & 0 & 0
\end{bmatrix} \quad \text{and} \quad
C:=\begin{bmatrix}
0 & 0 & 0 \\
0 & 1 & 1 \\
1 & 1 & 0
\end{bmatrix}.$$
Using the identities $\forall x \in \F_2, \; x+x=0 \; \text{and} \; x^2=x$, a straightforward computation yields
$$\forall (x,y,z)\in \F_2^3, \quad \det(I_3+x\,A+y\,B+z\,C)=1.$$
Therefore the 3-dimensional subspace $V:=\Vect(A,B,C)$ has a trivial spectrum.
The fact that $A+B$ is non-singular shows however that $V$ is irreducible.
If $V$ were reducible indeed, then there would exist a $1$-dimensional subspace $W$ of
$\Mat_2(\F_2)$ such that $V \simeq \{0\} \vee W$ or $V \simeq W \vee \{0\}$, and in both cases
every matrix of $V$ would be singular.

\vskip 2mm
The classification of the irreducible maximal subspaces of $\Mat_n(\F_2)$ with a trivial spectrum thus remains
an unresolved issue.

\end{document}